\documentclass[preprint,12pt,numbers]{elsarticle}
\usepackage{amsmath,amssymb,amsfonts,amsthm,mathrsfs}
\usepackage{graphicx,booktabs,array,bm,enumitem}
\usepackage[english]{babel}
\usepackage[colorlinks,linkcolor=blue,citecolor=blue,urlcolor=blue]{hyperref}
\biboptions{sort&compress}
\newtheorem{theorem}{Theorem}[section]
\newtheorem{lemma}[theorem]{Lemma}
\newtheorem{proposition}[theorem]{Proposition}
\newtheorem{corollary}[theorem]{Corollary}
\theoremstyle{definition}
\newtheorem{definition}[theorem]{Definition}
\newtheorem{assumption}[theorem]{Assumption}
\theoremstyle{remark}
\newtheorem{remark}[theorem]{Remark}
\numberwithin{equation}{section}
\allowdisplaybreaks[2]
\AtBeginDocument{\lefthyphenmin=3\righthyphenmin=3}
\hypersetup{
 pdftitle={Global well-posedness of the energy-critical nonlinear Schrodinger equation on the three-dimensional torus with additive Levy noise},
 pdfauthor={Bin Liu, Fan Xu},
 pdfsubject={Research manuscript, 14 October 2026}
}
\journal{Research manuscript}

\begin{document}
\begin{frontmatter}
\title{Global well-posedness of the energy-critical nonlinear
Schr\"odinger equation on the three-dimensional torus
with additive L\'evy noise}
\author[a,b]{Bin Liu}
\ead{binliu@mail.hust.edu.cn}
\author[a,b]{Fan Xu}
\ead{2026510746@hust.edu.cn}
\affiliation[a]{organization={School of Mathematics and Statistics,
Huazhong University of Science and Technology},city={Wuhan},
postcode={430074},state={Hubei},country={P. R. China}}
\affiliation[b]{organization={Hubei Key Laboratory of Engineering
Modeling and Scientific Computing, Huazhong University of Science
and Technology},city={Wuhan},postcode={430074},state={Hubei},
country={P. R. China}}
\begin{abstract}
This paper establishes global well-posedness for the defocusing
energy-critical nonlinear Schr\"odinger equation on the
three-dimensional torus with a class of additive L\'evy noises.
For arbitrary deterministic initial data in the energy space,
we construct probabilistically strong solutions that are
pathwise unique in the critical solution class. Continuous
dependence on initial data, convergence of finite-activity
approximations, and a vanishing-noise limit hold on finite
time intervals. When the effective noise is nonzero,
quantitative bounds for mass exit times imply almost-sure
unboundedness of the mass and exclude free scattering.
The expected energy then admits a cubic lower bound in time.
The analysis couples pathwise critical perturbation theory with
energy estimates adapted to Poisson compensation and stopping-time
arguments based on the exact mass balance.
\end{abstract}
\end{frontmatter}
\clearpage

\section{Introduction and main results}\label{sec:intro}

The deterministic defocusing quintic nonlinear Schr\"odinger
equation on the three-dimensional torus admits a global Cauchy
theory at the energy regularity $H^1$. An additive random forcing
changes both mass and energy. Additive L\'evy forcing leads
naturally to c\`adl\`ag solutions, while the nonlinear Duhamel
term remains continuous. These two
features require separate treatment: the energy estimate must
retain the cancellation provided by compensation, while the
critical perturbation theory must accept a discontinuous external
path.

In this paper we consider
\begin{equation}\label{eq:snls}
 \begin{split}
 du(t)&=i\bigl(\Delta u(t)-|u(t)|^4u(t)\bigr)\,dt
                 +\int_{Z_0}Gz\,\widetilde N(dt,dz),\\
 u(0)&=u_0\in H^1(\mathbb T^3),
 \qquad \mathbb T^3=(\mathbb R/2\pi\mathbb Z)^3.
 \end{split}
\end{equation}
Here $\widetilde N$ is a compensated Poisson random measure and
$G$ maps a finite-dimensional mark space into $H^3(\mathbb T^3)$.
The noise is additive and is interpreted as a square-integrable
martingale. The precise assumptions are given below. The
Hamiltonian drift is defocusing, and its exponent is critical for
the energy scaling of the corresponding equation on $\mathbb R^3$.

We construct a global solution on the prescribed stochastic
basis and prove pathwise uniqueness and convergence of the
small-jump approximations. The spatial regularity of the
noise allows the stochastic
convolution to be placed in the deterministic forcing norm by
a Sobolev estimate, leaving the nonlinear critical analysis in
the continuous remainder. The construction applies to arbitrary
deterministic $H^1$ initial data and noise intensities.


\subsection{Related work}
\label{subsec:related-work}

The deterministic theory provides both the reference flow and the
quantitative estimates used in this paper. The Fourier restriction
method of Bourgain~\cite{Bourgain1993} initiated the modern periodic
Cauchy theory. On compact manifolds, the Strichartz estimates of
Burq, G\'erard and Tzvetkov~\cite{BGT2004} reflect the role of the
underlying geometry in the dispersive theory. For the quintic
equation on $\mathbb T^3$, Herr, Tataru
and Tzvetkov~\cite{HTT2011} established the small-data energy-critical
theory using refined multilinear estimates and atomic function
spaces. Ionescu and Pausader~\cite{IP2012} then proved global
well-posedness for arbitrary defocusing $H^1$ data. Their argument
uses the Euclidean critical dynamics, for which the large-data
three-dimensional theorem is due to Colliander, Keel, Staffilani,
Takaoka and Tao~\cite{CKSTT2008}. The endpoint linear estimates of
Keel and Tao~\cite{KeelTao1998} and the expositions
\cite{Cazenave2003,Tao2006} supply general background for the
Euclidean local and perturbative theories.

Several subsequent developments clarify the periodic estimates.
The decoupling theorem of Bourgain and Demeter~\cite{BD2015} yields
the relevant periodic restriction estimates, and Killip and
Vi\c{s}an~\cite{KV2016} obtain scale-invariant Strichartz estimates
on rectangular tori. Strunk~\cite{Strunk2015} extends the
three-dimensional global defocusing theory to rectangular tori.
More recently, Kwak~\cite{Kwak2026} establishes global
well-posedness of the energy-critical defocusing equation in
general dimensions, with new conclusions for $d\ge5$. We work
on the standard three-dimensional torus and use the quantitative
three-dimensional theory. The $U^p$ and $V^p$ framework underlying
our spaces is developed in Hadac, Herr and Koch~\cite{HHK2009};
the corrections in~\cite{HHK2010} are relevant to its endpoint
conventions. The deterministic global theorem supplies unit-interval
bounds, which we combine by finite concatenation in the
whole-interval atomic norm defined below.

For Wiener-driven equations, the mass- and energy-space theories
of de Bouard and Debussche~\cite{dBD1999,dBD2003} are foundational.
The rescaling method of Barbu, R\"ockner and
Zhang~\cite{BRZ2014,BRZ2016} treats linear multiplicative noise by
transforming the equation into a random dispersive equation.
Stochastic Strichartz estimates were developed on compact
manifolds by Brze\'{z}niak and Millet~\cite{BM2014}, and used for
nonlinear Stratonovich coefficients on Euclidean space by
Hornung~\cite{Hornung2018}. A complementary compactness approach
gives energy-space martingale solutions in the work of
Brze\'{z}niak, Hornung and Weis~\cite{BHW2019}, with
three-dimensional compact-manifold uniqueness studied
in~\cite{BHW2022}.

Critical Gaussian problems require additional control beyond an
energy estimate. Oh and Okamoto~\cite{OO2020} prove global
well-posedness for additive-noise mass-critical and energy-critical
equations on Euclidean space. Fan and Xu~\cite{FX2021} establish
the one-dimensional mass-critical multiplicative-noise theory at
$L^2$ regularity, and their Wong--Zakai theorem~\cite{FXWZ2021}
addresses stability under approximation of the driving signal.
For linear multiplicative Wiener noise, Zhang~\cite{Zhang2023}
develops global mass- and energy-critical theories; pathwise
Strichartz and local smoothing estimates are treated
in~\cite{Zhang2022}. Scattering for multiplicative Wiener
equations is studied by Herr, R\"ockner and Zhang~\cite{HRZ2019}.
Their results concern long-time behavior on Euclidean domains. Below the
energy regularity, Cheung, Li and Oh~\cite{CLO2021} adapt the
almost-conservation method to a different, energy-subcritical
additive Gaussian problem. Recent focusing results of Millet
and Roudenko~\cite{MR2026} concern quantitative existence times
and probabilities in critical and supercritical regimes; the
defocusing sign used here supplies a different coercive mechanism.

The closest periodic Gaussian precedents are
Cheung--Mosincat~\cite{CM2019} and
Li--Okamoto--Tao~\cite{LOT2025}. The former treat additive and
multiplicative forcing on tori at scaling-subcritical
regularities. The latter establish global well-posedness for the
additively forced defocusing quintic equation on $\mathbb T^3$
at $H^1$ regularity, with uniqueness in the stochastic
convolution plus an atomic solution space. We use the
deterministic multilinear estimates developed in their argument.
The $H^3$ regularity of our spatial profiles places the
discontinuous stochastic convolution in the forcing space
through a Sobolev estimate.

Jump-driven Schr\"odinger equations have an established
well-posedness theory. De Bouard and
Hausenblas~\cite{dBH2019} construct martingale solutions,
and de Bouard, Hausenblas and Ondrej\'at~\cite{dBHO2019}
study pathwise uniqueness and strong solutions.
Brze\'{z}niak, Hornung and Manna~\cite{BHM2020} treat
energy-subcritical conservative Marcus equations in a framework
covering compact manifolds and bounded domains. Their
construction uses tightness and a martingale formulation.
Brze\'{z}niak, Liu and Zhu~\cite{BLZ2021} prove stochastic
Strichartz estimates for Poisson convolutions and obtain strong
well-posedness for Euclidean equations with nonlinear Marcus
noise in the mass-subcritical range. Wang, Zhai and
Zhu~\cite{WZZ2023} treat additive jump noise in
$L^2(\mathbb R^d)$ in that range. Large deviations for
Marcus-driven equations are studied by Zhu, Liu and
Zhai~\cite{ZLZ2026}. The monographs
\cite{Applebaum2009,PeszatZabczyk2007} provide the underlying
Poisson integration and Hilbert-space stochastic evolution
theory. The present energy-critical problem additionally
requires a critical spacetime continuation argument.

The present authors, together with Yang and Zhang,
considered the energy-critical equation on $\mathbb R^3$ with
conservative linear Marcus noise in~\cite{LXYZ2026}. That work
constructs global strong solutions, proves finite-activity
approximation, and obtains scattering under finite total
temporal noise intensity. It allows bounded Lipschitz
spatial profiles and centered finite-second-moment jump measures
without the bounded-support and symmetry restrictions imposed
here. The present equation has
a different jump rule and a different spatial geometry.
A Marcus phase jump preserves mass and potential energy;
an additive jump $u\mapsto u+Gz$ does not. On the torus,
Euclidean global dispersive norms must be replaced by local
periodic atomic estimates.

We combine the periodic quintic critical theory for a
discontinuous forcing path with a cutoff-uniform mass--energy
estimate for additive jumps and a convergence argument that
retains the prescribed Poisson random measure.
The deterministic forced equation is constructed in a continuous
remainder space, while all jumps reside in the stochastic
convolution. Finite-activity solutions can be concatenated from
the deterministic global flow; the uniform estimates and the
forced stability theorem then remove the jump cutoff. This
also proves continuity in probability with respect to the
initial datum. Uniqueness holds in the affine critical solution
class of Definition~\ref{def:solution} under
Assumption~\ref{ass:noise}. We also prove a quantitative exit
bound and almost-sure unboundedness of the mass when the
effective noise is nonzero. Ordinary free scattering is then
impossible; Section~\ref{sec:dynamics} distinguishes this
stochastic obstruction from the deterministic periodic one.

\subsection{Stochastic framework and solution class}
\label{subsec:framework}

Spatial integrals use normalized Haar measure,
\[
 \int_{\mathbb T^3}h(x)\,dx
       =(2\pi)^{-3}\int_{[-\pi,\pi]^3}h(x)\,dx.
\]
Unless another domain is specified, $L^p=L^p(\mathbb T^3;\mathbb C)$
and $\|h\|_p=\|h\|_{L^p}$. For $1\leq p<\infty$ and $p=\infty$,
respectively, these norms are
\[
 \|h\|_p=\left(\int_{\mathbb T^3}|h(x)|^p\,dx\right)^{1/p},
 \qquad
 \|h\|_\infty=\mathop{\rm ess\,sup}_{x\in\mathbb T^3}|h(x)|.
\]
Spatial integrals with the domain or measure omitted use the same
normalized measure. We write $\overline h$ for complex conjugation,
$\nabla=(\partial_{x_1},\partial_{x_2},\partial_{x_3})$, and
$\Delta=\sum_{j=1}^3\partial_{x_j}^2$. The complex pairing
\[
 \langle h,k\rangle_{L^2}
       =\int_{\mathbb T^3}h(x)\overline{k(x)}\,dx
\]
is linear in its first entry; its subscript is omitted when the
underlying space is clear.
For a smooth periodic function $h$, write
\[
 \widehat h(\xi)=\int_{\mathbb T^3}e^{-ix\cdot\xi}h(x)\,dx,
 \qquad \langle\xi\rangle=(1+|\xi|^2)^{1/2},\quad \xi\in\mathbb Z^3.
\]
For a time-dependent function, $\widehat h(t,\xi)$ denotes the
Fourier coefficient in the spatial variable at the fixed time $t$.
For $s\in\mathbb R$, the space $H^s=H^s(\mathbb T^3;\mathbb C)$
is the completion of smooth periodic functions under the norm
\[
 \|h\|_{H^s}^2=\sum_{\xi\in\mathbb Z^3}
                 \langle\xi\rangle^{2s}|\widehat h(\xi)|^2.
\]
In particular, $H^0=L^2$. The space $H^{-1}$ is also identified
with the continuous conjugate-linear functionals on $H^1$ through
the extension of this pairing, with
\[
 \|h\|_{H^{-1}}=
 \sup_{\|\phi\|_{H^1}\leq1}|\langle h,\phi\rangle_{H^{-1},H^1}|.
\]
For $1\leq p\leq\infty$, $W^{1,p}=W^{1,p}(\mathbb T^3;\mathbb C)$
consists of the $L^p$ functions whose first weak derivatives belong
to $L^p$, with norm $\|h\|_{W^{1,p}}=\|h\|_p+\|\nabla h\|_p$.
Let $S(t)=e^{it\Delta}$ and $\mathcal N(w)=|w|^4w$; equivalently,
$\widehat{S(t)h}(\xi)=e^{-it|\xi|^2}\widehat h(\xi)$.

For a Banach space $\mathcal B$ and an interval $I$, $L^q(I;\mathcal B)$
is the space of strongly measurable functions, identified almost
everywhere, with norm
\[
 \|f\|_{L^q(I;\mathcal B)}=
 \begin{cases}
  \left(\displaystyle\int_I\|f(t)\|_{\mathcal B}^q\,dt\right)^{1/q},
       &1\leq q<\infty,\\[4pt]
  \mathop{\rm ess\,sup}_{t\in I}\|f(t)\|_{\mathcal B},&q=\infty.
 \end{cases}
\]
Strong measurability means almost-everywhere approximation by
$\mathcal B$-valued measurable simple functions. We use the notation
$L_t^qL_x^r(I\times\mathbb T^3)=L^q(I;L^r)$ and
$L_t^qH_x^s(I\times\mathbb T^3)=L^q(I;H^s)$. When the interval is
fixed, $L^qH^s$ and $L_t^qH_x^s$ abbreviate $L^q(I;H^s)$; in
particular, $\|f\|_{L^\infty H^3}$ is the essential supremum in
time of $\|f(t)\|_{H^3}$. The norm of $L^p(I\times\mathbb T^3)$
uses the product of Lebesgue time measure and normalized spatial
measure. The sequence space $\ell^p$ uses counting measure, so that
$\|(a_j)\|_{\ell^p}=(\sum_j|a_j|^p)^{1/p}$ for $p<\infty$ and
$\|(a_j)\|_{\ell^\infty}=\sup_j|a_j|$.

The space $C(I;\mathcal B)$ consists of continuous $\mathcal B$-valued
functions; on a compact interval its norm is
$\|f\|_{C(I;\mathcal B)}=\sup_{t\in I}\|f(t)\|_{\mathcal B}$.
We also write $C_tH_x^s(I)=C(I;H^s)$ and, when the interval is
fixed, $CH^s=C(I;H^s)$.
The space $D([0,T];\mathcal B)$ consists of right-continuous
functions with left limits, also called c\`adl\`ag functions.
For such a path, write
\[
 X(t-)=\lim_{s\uparrow t}X(s),\qquad
 \Delta_tX=X(t)-X(t-)\quad(t>0).
\]
Thus $\Delta_t$ denotes a temporal jump and $\Delta$ the spatial
Laplacian. Temporal $L^\infty$ norms use the essential supremum,
whereas uniform path distances use the supremum over the indicated
closed interval.

For Banach spaces $\mathcal B_1,\mathcal B_2$,
$\mathcal L(\mathcal B_1,\mathcal B_2)$ denotes the bounded linear
maps, with operator norm
$\|A\|_{\mathcal L(\mathcal B_1,\mathcal B_2)}
=\sup_{\|h\|_{\mathcal B_1}\leq1}\|Ah\|_{\mathcal B_2}$.
The identity operator is denoted by $\mathrm{Id}$, and
$\mathcal B_1\hookrightarrow\mathcal B_2$ denotes a continuous
embedding. Differentiation of maps on complex spaces is understood
over the real scalars. For a twice Fr\'echet differentiable map $\Phi$,
\begin{align*}
 D\Phi(w)[h]&=\left.\frac{d}{dr}\Phi(w+rh)\right|_{r=0},\\
 D^2\Phi(w)[h,k]&=
 \left.\frac{\partial^2}{\partial r\,\partial s}
              \Phi(w+rh+sk)\right|_{r=s=0},\qquad r,s\in\mathbb R.
\end{align*}
We write $\mathbf1_A$ for the indicator of a set $A$,
$a\wedge b=\min\{a,b\}$, and $\lceil a\rceil$ for the least
integer greater than or equal to $a$. The notation $|I|$ denotes
the length of a time interval, while $|z|$ is the Euclidean norm
of a mark $z\in\mathbb R^m$. Finite time horizons $T$ are always
strictly positive. For stopping times, we use $\inf\varnothing=\infty$.

For nonnegative quantities $A,B$, the relation $A\lesssim B$
means $A\leq CB$ for a finite deterministic constant depending
only on the fixed torus and the fixed frequency decomposition.
A subscript, as in $\lesssim_{\mathfrak p}$, records additional
dependence on the parameters $\mathfrak p$. Generic constants may
change from line to line. Constants fixed in a statement or used
to choose a time step are displayed explicitly.

Fix a complete stochastic basis
$(\Omega,\mathcal F,(\mathcal F_t)_{t\ge0},\mathbb P)$ satisfying
the usual conditions: the filtration is right-continuous and
$\mathcal F_0$ contains all $\mathbb P$-null sets. Here $\Omega$
is the sample space, $\mathcal F$ its sigma field, $\mathbb P$ the
probability measure, and $\mathbb E$ expectation under $\mathbb P$.
For a strongly measurable random variable $X$ with values in
$\mathcal B$, $\|X\|_{L^q(\Omega;\mathcal B)}
=(\mathbb E\|X\|_{\mathcal B}^q)^{1/q}$ when $1\leq q<\infty$.
Almost-sure statements are with respect to $\mathbb P$.
Let $m\ge1$ be an integer and
$Z_0=\{z\in\mathbb R^m:0<|z|\le1\}$. The following assumptions
are in force throughout the paper.

\begin{assumption}\label{ass:noise}
The measure $\nu$ is a Borel L\'evy measure on $Z_0$ such that,
for every Borel set $A\subset Z_0$,
\begin{equation}\label{eq:noiseass}
 \nu(A)=\nu(-A),\qquad
 \nu(Z_0)=\infty,\qquad
 m_2:=\int_{Z_0}|z|^2\,\nu(dz)<\infty.
\end{equation}
The random measure $N$ is an $(\mathcal F_t)$-Poisson random
measure with compensator $dt\,\nu(dz)$, and
$\widetilde N=N-dt\,\nu$. Its increments after each deterministic
time are independent of the past sigma field. The coefficient is
the real-linear map
\begin{equation}\label{eq:profiles}
 Gz=\sum_{k=1}^m z_k f_k,
 \qquad f_k\in H^3(\mathbb T^3;\mathbb C).
\end{equation}
\end{assumption}

Set
\begin{equation}\label{eq:noise-sizes}
 \begin{aligned}
 q_j&=\int_{Z_0}\|Gz\|_{H^j}^2\,\nu(dz),\qquad j\in\{0,1,3\},\\
 q_3(r)&=\int_{\{z\in Z_0:|z|\le r\}}\|Gz\|_{H^3}^2\,\nu(dz),
 \qquad 0<r\leq1.
 \end{aligned}
\end{equation}
All these quantities are finite, and $q_3(r)\to0$ as $r\downarrow0$.
Constants may depend on $G$, $m_2$, and the fixed torus and are
uniform in the small-jump cutoff. The arguments below also apply
when $\nu(Z_0)<\infty$.

The process
\[
 L(t)=\int_{(0,t]}\int_{Z_0}z\,\widetilde N(ds,dz)
\]
is defined in the square-integrable martingale sense. There is no
Gaussian component and no additional deterministic drift.
The compensated integral is the square-integrable limit of its
truncations, allowing $\int_{Z_0}|z|\,\nu(dz)=\infty$.
This convention fixes the centering of the noise in \eqref{eq:snls}.

The stochastic convolution is
\begin{equation}\label{eq:convolution}
 \Psi(t)=\int_{(0,t]}\int_{Z_0}S(t-s)Gz\,\widetilde N(ds,dz).
\end{equation}
It has an adapted c\`adl\`ag version with values in $H^3$.
The integral may equivalently be defined through the
$H^3$-valued martingale $S(-t)\Psi(t)$.

\begin{definition}\label{def:solution}
A solution of \eqref{eq:snls} is an adapted process $u$ with
c\`adl\`ag $H^1$ paths such that, for every $T<\infty$,
\begin{equation}\label{eq:solution-class}
 v:=u-\Psi\in X^1([0,T])
\end{equation}
almost surely, and
\begin{equation}\label{eq:mild}
 u(t)=S(t)u_0-i\int_0^t S(t-s)\mathcal N(u(s))\,ds+\Psi(t)
 \qquad(0\le t\le T).
\end{equation}
The space $X^1(I)$ is the continuous atomic restriction space
defined in Section~\ref{sec:deterministic}. In particular,
$X^1(I)\hookrightarrow C(I;H^1)$. The identity
\eqref{eq:mild} is initially understood in $H^{-1}$, and its
nonlinear integral in fact has a continuous $H^1$ representative.
\end{definition}

Indeed, the Sobolev embedding $H^1\hookrightarrow L^6$ gives
$\mathcal N:H^1\to H^{-1}$, locally Lipschitz on bounded sets.
A c\`adl\`ag path on a compact interval is bounded, so its
nonlinearity belongs to $L^1([0,T];H^{-1})$. Using $u(s-)$
instead of $u(s)$ in the time integral gives the same result.
All stochastic integrals with state-dependent integrands below
use the predictable left limit.

Pathwise uniqueness in this paper means that two solutions in
Definition~\ref{def:solution}, with the same initial datum and
Poisson random measure on the same stochastic basis, are
indistinguishable. A probabilistically strong solution means a
solution constructed on this given basis.

\subsection{Main theorem}

Define the mass, energy, and their positive sum by
\begin{equation}\label{eq:energyfunctional}
 \begin{split}
 M(w)&=\|w\|_{L^2}^2,\qquad
 E(w)=\frac12\|\nabla w\|_{L^2}^2+\frac16\|w\|_{L^6}^6,\\
 F(w)&=1+M(w)+E(w).
 \end{split}
\end{equation}
The distinction between $M(w)$ and the spacetime quantity $M(I)$
introduced later is indicated by their arguments. We have
$\|w\|_{H^1}^2\le2F(w)$ and $F(w)\ge1$.

For $0<\varepsilon<1$, write
$Z_\varepsilon=\{z:\varepsilon<|z|\le1\}$ and let
$u^\varepsilon$ solve
\begin{equation}\label{eq:cutoff-model}
 du^\varepsilon
  =i\bigl(\Delta u^\varepsilon-\mathcal N(u^\varepsilon)\bigr)dt
    +\int_{Z_\varepsilon}Gz\,\widetilde N(dt,dz),
 \qquad u^\varepsilon(0)=u_0.
\end{equation}
The same cutoff defines $\Psi^\varepsilon$, and we set
$v^\varepsilon=u^\varepsilon-\Psi^\varepsilon$.
All cutoffs use the original random measure $N$.

\begin{theorem}\label{thm:main}
Under Assumption~\ref{ass:noise}, every deterministic
$u_0\in H^1(\mathbb T^3)$ gives a global probabilistically strong
solution of \eqref{eq:snls}. The following conclusions hold.
\begin{enumerate}[label=\textup{(\roman*)},leftmargin=*]
\item For every finite $T$, almost surely
\[
 u\in D([0,T];H^1),\qquad u-\Psi\in X^1([0,T]).
\]
Pathwise uniqueness holds in this class.
\item On every finite time interval,
\begin{equation}\label{eq:main-moment}
 \sup_{0<\varepsilon<1}\mathbb E\sup_{0\leq t\leq T}
       F(u^\varepsilon(t))^2
 +\mathbb E\sup_{0\leq t\leq T}F(u(t))^2
       \lesssim_{T,G,m_2} F(u_0)^2.
\end{equation}
\item As $\varepsilon\downarrow0$,
\begin{equation}\label{eq:main-cutoff}
 \sup_{0\leq t\leq T}\|u^\varepsilon(t)-u(t)\|_{H^1}
 +\|v^\varepsilon-v\|_{X^1([0,T])}
           \longrightarrow0
 \quad\hbox{in probability}.
\end{equation}
\item If deterministic data $u_{0,n}$ tend to $u_0$ in $H^1$
and $u_n$ denotes the solution driven by the same noise, then
\begin{equation}\label{eq:main-data}
 \begin{split}
 &\sup_{0\leq t\leq T}\|u_n(t)-u(t)\|_{H^1}\\
 &\quad+\|(u_n-\Psi)-(u-\Psi)\|_{X^1([0,T])}
           \longrightarrow0
 \quad\hbox{in probability}.
 \end{split}
\end{equation}
\end{enumerate}
\end{theorem}

The convergence in \eqref{eq:main-cutoff} uses the uniform
distance on c\`adl\`ag paths, which is stronger than convergence
in the Skorokhod $J_1$ topology. For paths $x_n,x\in D([0,T];H^1)$,
convergence $x_n\to x$ in $J_1$ means that there are increasing
homeomorphisms $\lambda_n:[0,T]\to[0,T]$ such that
\[
 \sup_{0\leq t\leq T}|\lambda_n(t)-t|
 +\sup_{0\leq t\leq T}\|x_n(\lambda_n(t))-x(t)\|_{H^1}\longrightarrow0.
\]
Uniform convergence is possible here because the approximations
retain exactly the same large jumps,
at the same times and with the same marks. The $X^1$ distance is
applied to the continuous remainders.

The proof also gives an explicit localized error estimate.
For each $R\ge1$, there is a finite deterministic constant $C_T(R)$,
independent of the cutoffs, such that the following estimate holds
for all $0<\varepsilon,\eta<1$. Set
\[
 d_T(\varepsilon,\eta)
  =\sup_{0\leq t\leq T}\|u^\varepsilon(t)-u^\eta(t)\|_{H^1}
      +\|v^\varepsilon-v^\eta\|_{X^1([0,T])},
\]
Then
\begin{equation}\label{eq:intro-error}
 \mathbb P\bigl(d_T(\varepsilon,\eta)>a\bigr)
 \le \frac{C_{T,u_0,G,\nu}}{R^2}
       +\frac{4T C_T(R)^2}{a^2}\,
           q_3\bigl(\max(\varepsilon,\eta)\bigr)
 \qquad(a>0).
\end{equation}
The cutoff limit is taken with $R$ fixed, followed by
$R\to\infty$. For each fixed $R$, the deterministic large-data
theory supplies a finite constant $C_T(R)$.

\begin{remark}
\label{rem:deterministic-extension}
When $Gz=0$ for $\nu$-almost every $z$, equation
\eqref{eq:snls} reduces to the deterministic defocusing quintic
equation, and Theorem~\ref{thm:main} recovers the arbitrary-data
$H^1$ global theory of Ionescu and Pausader~\cite{IP2012}.
For nonzero noise, it extends this Cauchy theory to additive
L\'evy perturbations satisfying Assumption~\ref{ass:noise}.
Corollary~\ref{cor:vanishing} establishes
the corresponding vanishing-noise limit on every finite time interval.
\end{remark}

The same construction also gives a long-time statement beyond
well-posedness. If $q_0>0$, Proposition~\ref{dyn:massescape}
shows that every bounded $L^2$ ball is exited in finite expected
time and that $\limsup_{t\to\infty}\|u(t)\|_2=\infty$
almost surely. These conclusions follow from the stopped
pathwise mass balance. Corollary~\ref{dyn:stoch-nonscattering} consequently rules
out convergence to any $L^2$-bounded reference trajectory,
including a free orbit with a random terminal profile.

\subsection{Proof strategy and organization}
\label{subsec:strategy}

The proof combines deterministic perturbation estimates with
probabilistic approximation.

First, write $u=v+\Psi$. For a fixed forcing path $f$ on
$I=[a,b]\subset[0,\infty)$, the
remainder equation is
\[
 v(t)=S(t-a)v(a)-i\int_a^t
             S(t-s)\mathcal N(v(s)+f(s))\,ds.
\]
The mixed estimates of Li, Okamoto and Tao \cite{LOT2025}
hold for measurable forcing in the space $Z(I)$ defined below.
We combine them with the deterministic unit-time large-data
bound of Ionescu and Pausader \cite{IP2012}. A comparison with
the deterministic solution, followed by a finite subdivision
of its critical spacetime norm, gives a restart time depending
only on bounds for $\|v(a)\|_{H^1}$ and
$\|f\|_{L^\infty H^3}$.

Second, a finite-activity approximation is constructed by the
deterministic flow between jumps and by the translation
$w\mapsto w+Gz$ at a jump. Mass and energy are conserved between
jumps. A Taylor expansion of their exact increments, together
with the cancellation of the first-order compensator, gives
a cutoff-independent maximal second moment of $F$. The jump
product formula is applied to the finite-activity process.

Third, in the interaction representation the convolution is
a Hilbert-space martingale. Doob's inequality controls its
small-jump tail in the mean square of the uniform $H^3$ path
distance on every finite time interval. On events
where the energies and convolutions are bounded, deterministic
stability turns this tail estimate into the Cauchy bound
\eqref{eq:intro-error}. The limit is taken on the original
stochastic basis. A countable subsequence and completeness give
the pathwise construction, and uniqueness identifies all cutoff
limits. The necessary measurability of the remainder norms is
checked through the finite-jump parameterization.

Section~\ref{sec:deterministic} proves the deterministic forcing
results, including the difference estimates and the order of
choices of all restart parameters. Section~\ref{sec:probability}
constructs the stochastic solution and proves
Theorem~\ref{thm:main}. Section~\ref{sec:consequences} records
stronger convergence in the energy path norm and the finite-time
vanishing-noise limit. Section~\ref{sec:dynamics} proves the
pathwise mass balance, the quantitative exit estimate and the
absence of ordinary free scattering; it also explains the
distinct deterministic obstruction on the torus.

\section{Deterministic estimates for measurable perturbations}
\label{det:section}
\label{sec:deterministic}

We establish estimates for a prescribed measurable perturbation and a
continuous remainder in the atomic solution space.

Throughout this section, $\mathbb T^3=(\mathbb R/2\pi\mathbb Z)^3$ is
equipped with normalized Haar measure, $S(t)=e^{it\Delta}$, and
$\mathcal N(z)=|z|^4z$. For a finite interval $I=[a,b]$, $a<b$,
and prescribed initial value $v_a\in H^1$, we consider
\begin{equation}\label{det:equation}
 v(t)=S(t-a)v_a-i\int_a^t S(t-s)\mathcal N(v(s)+f(s))\,ds.
\end{equation}
The construction and stability estimates below apply to strongly
measurable perturbations $f\in L^\infty(I;H^3)$. Their constants
are uniform over bounded subsets of this space.

\subsection{Atomic spaces and interval conventions}
\label{det:spaces}

A scalar $U^2$-atom is a step function
\[
 \begin{gathered}
 a(t)=\sum_{j=1}^{J}\mathbf1_{[t_{j-1},t_j)}(t)c_{j-1},
 \qquad c_0=0,\\
 -\infty=t_0<t_1<\cdots<t_J=\infty,
 \qquad \sum_{j=0}^{J-1}|c_j|^2=1.
 \end{gathered}
\]
Here $J$ is a positive integer and $c_j\in\mathbb C$.
The space $U^2(\mathbb R;\mathbb C)$ consists of sums
$u=\sum_{k\geq1}\lambda_k a_k$, where $a_k$ are $U^2$-atoms,
$\lambda_k\in\mathbb C$, and $\sum_k|\lambda_k|<\infty$.
Its norm is
\[
 \|u\|_{U^2(\mathbb R;\mathbb C)}
 =\inf\left\{\sum_{k\geq1}|\lambda_k|:
     u=\sum_{k\geq1}\lambda_k a_k,\quad a_k\text{ are }U^2\text{-atoms}\right\}.
\]
We abbreviate this norm by $\|u\|_{U^2}$. These series converge
uniformly in time and in the atomic norm. We use the
standard right-continuous representatives. The atomic construction and
its restriction properties are developed in \cite{HHK2009,HHK2010}; the
Schr\"odinger spaces below follow \cite{HTT2011}.

For $s\in\mathbb R$ and $w:\mathbb R\to H^s(\mathbb T^3)$, set
\begin{equation}\label{det:Xglobal}
 \|w\|_{\widehat X^s}^2
 =\sum_{\xi\in\mathbb Z^3}\langle\xi\rangle^{2s}
  \|e^{it|\xi|^2}\widehat w(t,\xi)\|_{U^2(\mathbb R;\mathbb C)}^2,
 \qquad
 X^s=\widehat X^s\cap C(\mathbb R;H^s).
\end{equation}
Here $\widehat X^s$ consists of the functions for which the displayed
norm is finite. The continuous subspace is closed: convergence in
$\widehat X^s$ implies uniform convergence in $H^s$.
The space $X^s$ carries the inherited norm
$\|w\|_{X^s}=\|w\|_{\widehat X^s}$.
For a compact interval $I$, define
\begin{equation}\label{det:Xrestriction}
 X^s(I)=\{w|_I:w\in X^s\},\qquad
 \|v\|_{X^s(I)}=
 \inf_{w\in X^s,\ w|_I=v}\|w\|_{\widehat X^s}.
\end{equation}
The quotient construction gives a Banach space, continuously embedded
in $C(I;H^s)$. Closed endpoints are interpreted through this continuous
representative. Restriction to a smaller interval cannot increase the
norm. For an interval $J$, membership in $X^s_{\mathrm{loc}}(J)$
means membership in $X^s(I)$ on every compact subinterval $I\subset J$.
In particular, $v\in X^s_{\mathrm{loc}}([a,b))$ means
$v|_{[a,c]}\in X^s([a,c])$ for every $a<c<b$.
We shall use the cases $s=0$ and $s=1$.

For a scalar function $z:\mathbb R\to\mathbb C$, its total variation is
\[
 \operatorname{Var}_{\mathbb R}(z)
 =\sup_{\substack{J\geq1\\-\infty<t_0<\cdots<t_J<\infty}}
       \sum_{j=1}^J|z(t_j)-z(t_{j-1})|.
\]
On $[a,b]$, the quantity $\operatorname{Var}_{[a,b]}(z)$ is defined
by the same supremum over partitions $a=t_0<\cdots<t_J=b$.
The corresponding variation function is
$V_z(t)=\operatorname{Var}_{[a,t]}(z)$ for $a<t\leq b$, with
$V_z(a)=0$.

\begin{lemma}\label{det:variation}
Let $z:\mathbb R\to\mathbb C$ be continuous and of finite total
variation. Suppose that $z$ is zero to the left of a finite interval
and constant to its right. Then
\[
 \|z\|_{U^2}\leq\operatorname{Var}_{\mathbb R}(z).
\]
\end{lemma}

\begin{proof}
Choose $[a,b]$ so that $z=0$ on $(-\infty,a]$ and $z=z(b)$
on $[b,\infty)$, and set $V=\operatorname{Var}_{\mathbb R}(z)$.
If $V=0$, then $z=0$ and the assertion follows. Suppose $V>0$.
The variation function $V_z$ is continuous and nondecreasing.
For $n\geq0$, choose a nested partition
$a=t_{n,0}<\cdots<t_{n,J_n}=b$ at the levels
$k2^{-n}V$, $1\leq k<2^n$, of $V_z$. Then $J_n\leq2^n$ and
\[
 \operatorname{Var}_{[t_{n,j-1},t_{n,j}]}(z)\leq2^{-n}V.
\]
Let $z_n$ take the value $z(t_{n,j-1})$ on
$[t_{n,j-1},t_{n,j})$, be zero before $a$, and equal $z(b)$
on $[b,\infty)$. Its increment representation is
\[
 z_n(t)=\sum_{j=1}^{J_n}
  \bigl(z(t_{n,j})-z(t_{n,j-1})\bigr)
                  \mathbf1_{[t_{n,j},\infty)}(t).
\]
Each half-line indicator is a $U^2$-atom. Consequently,
\[
 \|z_n\|_{U^2}
 \leq\sum_{j=1}^{J_n}|z(t_{n,j})-z(t_{n,j-1})|
 \leq V,
 \qquad \|z_n-z\|_{L^\infty(\mathbb R)}\leq2^{-n}V.
\]
The step function $z_{n+1}-z_n$ vanishes outside $[a,b)$.
On the finer partition it has at most $2^{n+1}$ values,
each of modulus at most $2^{-n}V$. Normalizing these values
in $\ell^2$ and using the atomic definition gives
\[
 \|z_{n+1}-z_n\|_{U^2}
 \leq (2^{n+1})^{1/2}2^{-n}V
 \lesssim 2^{-n/2}V.
\]
The series of successive differences therefore converges in
$U^2$. Since $\|h\|_{L^\infty}\leq\|h\|_{U^2}$, its limit
coincides with the uniform limit $z$. Passing to the limit
in $\|z_n\|_{U^2}\leq V$ proves the claimed bound.
This completes the proof.
\end{proof}

We shall use the estimates
\begin{equation}\label{det:linear}
 \|v\|_{C(I;H^1)}\leq\|v\|_{X^1(I)},\qquad
 \|S(t-a)\phi\|_{X^1(I)}\leq\|\phi\|_{H^1}.
\end{equation}
For the second estimate, choose a continuous nondecreasing cutoff
$\chi$ that is zero on $(-\infty,a-1]$ and one on $[a,\infty)$.
The extension $\chi(t)S(t-a)\phi$ belongs to $X^1$.
Applying Lemma~\ref{det:variation} to each Fourier coefficient,
with $\operatorname{Var}_{\mathbb R}(\chi)=1$, bounds its norm by $\|\phi\|_{H^1}$.
These estimates also appear as the restriction-space embedding and
the free evolution estimate in \cite[Section 2.1 and Proposition
2.7]{LOT2025}.

For $I=[a,b]$ and $H\in L^1(I;H^{-1})$, define the
$H^{-1}$-valued Bochner integral, source class, and Duhamel norm by
\begin{equation}\label{det:source}
 \begin{aligned}
 \mathcal I_aH(t)&=\int_a^t S(t-s)H(s)\,ds,\\
 N^1(I)&=\{H\in L^1(I;H^{-1}):\mathcal I_aH\in X^1(I)\},\\
 \|H\|_{N^1(I)}&=\|\mathcal I_aH\|_{X^1(I)}.
 \end{aligned}
\end{equation}
For $H\in L^1(I;H^1)$, in particular,
\begin{equation}\label{det:L1source}
 \|H\|_{N^1(I)}\leq\|H\|_{L^1(I;H^1)}.
\end{equation}
Indeed, for $I=[a,b]$, extend $S(-t)\mathcal I_aH(t)$ by zero
before $a$ and by its value at $b$ after $b$. This extension is
continuous in $H^1$. Lemma~\ref{det:variation}, applied coefficientwise,
and Minkowski's inequality give
\[
 \|\mathcal I_aH\|_{X^1(I)}
 \leq\left(\sum_{\xi\in\mathbb Z^3}\langle\xi\rangle^2
       \left(\int_a^b|\widehat H(s,\xi)|\,ds\right)^2\right)^{1/2}
 \leq\int_a^b\|H(s)\|_{H^1}\,ds.
\]
For the quintic sources below, the integral initially has an
$H^{-1}$-valued meaning. The nonlinear estimates will show that it
has a continuous $H^1$-valued representative in $X^1(I)$.
Statements about a polynomial with values in $N^1(I)$ can equivalently
be read after applying $\mathcal I_a$, with values in the Banach space
$X^1(I)$.

\begin{lemma}\label{det:concatenation}
There is an absolute constant $C_{\rm con}$ with the following
property. Suppose $a=t_0<\cdots<t_J=b$ and
$v_j\in X^1([t_{j-1},t_j])$ have matching endpoint values. The
continuous function $v$ obtained by concatenation belongs to
$X^1([a,b])$ and satisfies
\begin{equation}\label{det:concatbound}
 \|v\|_{X^1([a,b])}
 \leq C_{\rm con}\sum_{j=1}^J
       \|v_j\|_{X^1([t_{j-1},t_j])}.
\end{equation}
The constant is independent of $J$ and of the lengths of the
subintervals.
\end{lemma}

\begin{proof}
For a half-open interval $J$ and a scalar $U^2$-atom
$a=\sum_k c_k\mathbf1_{[s_k,s_{k+1})}$, its restriction
$\mathbf1_Ja$ is a step function whose nonzero coefficients
form a subcollection of the $c_k$. Splitting at the two
endpoints of $J$ gives
\[
 \|\mathbf1_Ja\|_{U^2}
 \leq\left(\sum_{k:[s_k,s_{k+1})\cap J\ne\varnothing}
                    |c_k|^2\right)^{1/2}\leq1.
\]
Atomic decompositions and then the Fourier definition imply
\[
 \|\mathbf1_Jh\|_{U^2}\leq\|h\|_{U^2},
 \qquad
 \|\mathbf1_Jw\|_{\widehat X^1}\leq\|w\|_{\widehat X^1}.
\]
For $\varepsilon>0$, choose extensions $w_j\in X^1$ of $v_j$
such that
\[
 \|w_j\|_{\widehat X^1}
 \leq\|v_j\|_{X^1([t_{j-1},t_j])}+\varepsilon/J.
\]
Choose a continuous nondecreasing $\chi$ equal to zero before
$a-1$ and to one on $[a,\infty)$. Define
\begin{align*}
 W(t)={}&\mathbf1_{(-\infty,a)}(t)\chi(t)S(t-a)v(a)\\
 &+\sum_{j=1}^J\mathbf1_{[t_{j-1},t_j)}(t)w_j(t)
   +\mathbf1_{[b,\infty)}(t)S(t-b)v(b).
\end{align*}
The matching traces make $W$ continuous, and $W|_{[a,b]}=v$.
In the interaction representation, the first and last terms
have constant Fourier factors multiplied, respectively, by
$\chi-\mathbf1_{[a,\infty)}$ and $\mathbf1_{[b,\infty)}$.
Lemma~\ref{det:variation} gives $\|\chi\|_{U^2}\leq1$, so
\begin{align*}
 \|W\|_{\widehat X^1}
 &\leq2\|v(a)\|_{H^1}+\|v(b)\|_{H^1}
                    +\sum_{j=1}^J\|w_j\|_{\widehat X^1}\\
 &\leq4\sum_{j=1}^J\|v_j\|_{X^1([t_{j-1},t_j])}
       +\varepsilon.
\end{align*}
Thus $W\in X^1$. Taking the infimum over extensions and then
letting $\varepsilon\downarrow0$ proves
\eqref{det:concatbound}, with $C_{\mathrm{con}}=4$ as one
possible choice. This completes the proof.
\end{proof}

Fix a real radial function $\rho\in C_c^\infty(\mathbb R^3)$, where
$C_c^\infty$ denotes smooth compactly supported functions, such that
$0\leq\rho\leq1$, $\rho=1$ on $\{|\xi|\leq1\}$, and
$\rho=0$ on $\{|\xi|\geq2\}$. For dyadic integers
$N\in\{1,2,4,\ldots\}$, define the spatial Fourier multipliers
\[
 \begin{split}
  \widehat{P_1h}(\xi)&=\rho(\xi)\widehat h(\xi),\\
  \widehat{P_Nh}(\xi)&=
  [\rho(\xi/N)-\rho(2\xi/N)]\widehat h(\xi),\qquad N\geq2.
 \end{split}
\]
Thus $\sum_NP_N=\mathrm{Id}$, the frequency supports have finite
overlap, and $P_1$ includes the zero frequency. On time-dependent
functions these operators act only on the spatial variable. Put
$p_0=41/10$ and $p_1=100$. For any finite interval $I$, define
\begin{align}
 \|v\|_{Q_p(I)}^p
 &=\sum_{N\geq1}N^{5-p/2}
       \|P_Nv\|_{L^p(I\times\mathbb T^3)}^p,
       \qquad p\in\{p_0,p_1\},\label{det:Qp}\\
 \|v\|_{Q(I)}&=\|v\|_{Q_{p_0}(I)}+\|v\|_{Q_{p_1}(I)},
 \qquad
 \|v\|_{M(I)}=\|v\|_{X^1(I)}^{1/2}\|v\|_{Q(I)}^{1/2},\label{det:QM}\\
 \|f\|_{Z(I)}^2
 &=\sum_{N\geq1}N^2
       \|P_Nf\|_{L^{82}(I\times\mathbb T^3)}^2.
       \label{det:Z}
\end{align}
The quantities $Q_p(I)$, $Q(I)$, and $M(I)$ will be evaluated on
$v\in X^1(I)$; $\|v\|_{M(I)}$ is an auxiliary homogeneous quantity.
We define
\[
 Z(I)=\{f\in L^{82}(I;H^1):\|f\|_{Z(I)}<\infty\},
\]
with functions identified almost everywhere. For an interval $J$,
$f\in Z_{\mathrm{loc}}(J)$ means that $f|_I\in Z(I)$ for every
compact subinterval $I\subset J$.
All sums over $N$ are dyadic. On intervals of length at most
one, $Q$ is the weak critical norm denoted by $Z$ in
\cite{IP2012} and by $\widetilde X^1$ in \cite{LOT2025}; our $Z$
is instead the perturbation norm of \cite[(3.3)]{LOT2025}.
The exponent $82$ is $2p_0/(p_0-4)$.

The periodic frequency-localized Strichartz estimates in
\cite[Section 3]{HTT2011}, transferred to the atomic spaces, imply
\begin{equation}\label{det:Qembed}
 \|v\|_{Q(I)}\leq C_Q\|v\|_{X^1(I)},\qquad |I|\leq1.
\end{equation}
Indeed, for $p\in\{p_0,p_1\}$ the frequency-localized estimate is
\[
 \|P_Nv\|_{L^p(I\times\mathbb T^3)}
 \lesssim N^{3/2-5/p}\|P_Nv\|_{X^0(I)}.
\]
Since $5-p/2+p(3/2-5/p)=p$, one obtains
\begin{align*}
 \|v\|_{Q_p(I)}
 &\lesssim\left(\sum_N(N\|P_Nv\|_{X^0(I)})^p\right)^{1/p}\\
 &\leq\left(\sum_NN^2\|P_Nv\|_{X^0(I)}^2\right)^{1/2}
 \lesssim\|v\|_{X^1(I)}.
\end{align*}
The last bound follows by taking any whole-line extension of
$v$, using the finite overlap of the Fourier supports, and
then taking the infimum over extensions. Summing over the
two values of $p$ gives \eqref{det:Qembed}.

\begin{lemma}
\label{det:subdivision}
Let $v\in X^1(I)$, where $|I|\leq1$, and let $\gamma>0$.
There is a partition $I=\bigcup_{j=1}^J I_j$ such that
$\|v\|_{Q(I_j)}\leq\gamma$ and
\begin{equation}\label{det:partitioncount}
 J\leq1+\sum_{p\in\{p_0,p_1\}}
       \left(\frac{2\|v\|_{Q_p(I)}}{\gamma}\right)^p.
\end{equation}
The same assertion holds simultaneously for finitely many
functions, with the corresponding sums in the right-hand side.
A prescribed upper bound $\ell$ for the subinterval lengths may be
imposed by adding at most $\lceil |I|/\ell\rceil$ cuts.
\end{lemma}

\begin{proof}
Write $I=[a,b]$ and let $v_1,\ldots,v_L\in X^1(I)$, where
$L\geq1$. For $1\leq r\leq L$, $p\in\{p_0,p_1\}$, and a
Borel set $E\subset I$, define
\[
 \mu_{r,p}(E)=\sum_NN^{5-p/2}
       \int_{E\times\mathbb T^3}|P_Nv_r(t,x)|^p\,dx\,dt.
\]
These are finite measures absolutely continuous with respect
to time Lebesgue measure, with
$\mu_{r,p}(I)=\|v_r\|_{Q_p(I)}^p$.
Starting at $t_0=a$, choose $t_j$ to be the first time after
$t_{j-1}$ at which one of the quantities
$\mu_{r,p}([t_{j-1},t])$ reaches $(\gamma/2)^p$, or set
$t_j=b$ if no such time occurs before $b$. Continuity of these
quantities gives
\[
 \mu_{r,p}([t_{j-1},t_j])\leq(\gamma/2)^p
 \quad\hbox{for every }r,p,j.
\]
For each nonterminal interval choose one pair $(r,p)$ for
which equality holds, and let $n_{r,p}$ be the number of
intervals assigned to that pair. Their interiors are disjoint,
and the endpoints have zero measure. Hence
\begin{align*}
 n_{r,p}(\gamma/2)^p&\leq\mu_{r,p}(I),\\
 J-1&=\sum_{r=1}^L\sum_{p\in\{p_0,p_1\}}n_{r,p}\\
 &\leq\sum_{r=1}^L\sum_{p\in\{p_0,p_1\}}
       \left(\frac{2\|v_r\|_{Q_p(I)}}{\gamma}\right)^p.
\end{align*}
In particular, the construction terminates after finitely
many intervals, and on each interval $I_j=[t_{j-1},t_j]$,
\[
 \|v_r\|_{Q(I_j)}
 =\sum_{p\in\{p_0,p_1\}}\mu_{r,p}(I_j)^{1/p}
 \leq\gamma.
\]
Taking $L=1$ gives \eqref{det:partitioncount}.
To impose a length bound $\ell>0$, insert the grid points
$a+k\ell$ lying in $(a,b)$. There are at most
$\lceil|I|/\ell\rceil$ such points, and each resulting interval
$J'\subset I_j$ satisfies
\[
 \|v_r\|_{Q_p(J')}^p=\mu_{r,p}(J')
 \leq\mu_{r,p}(I_j),
 \qquad \|v_r\|_{Q(J')}\leq\|v_r\|_{Q(I_j)}\leq\gamma.
\]
This completes the proof.
\end{proof}

\subsection{The deterministic inputs and the perturbation norm}
\label{det:inputs}

We use the following two deterministic estimates.

\begin{proposition}\label{det:unforced}
For every $R\geq1$ there is a finite nondecreasing constant
$B_0(R)$ such that the solution of
\[
 i\partial_tw+\Delta w=\mathcal N(w),\qquad
 w(a)=w_a,\qquad \|w_a\|_{H^1}\leq R,
\]
exists globally, conserves mass and energy, and satisfies
\begin{equation}\label{det:unforcedbound}
 \|w\|_{X^1([a,a+1])}\leq B_0(R).
\end{equation}
The bound is independent of $a$ and of the particular datum in
the indicated ball.
\end{proposition}

\begin{proof}
The Sobolev embedding and the definitions of mass and energy give
\[
 \|w_a\|_2^2\leq R^2,\qquad
 \frac12\|\nabla w_a\|_2^2+\frac16\|w_a\|_6^6
 \lesssim R^2+R^6.
\]
We apply the global theorem of Ionescu and Pausader
\cite{IP2012} at these bounds. Their
Section 6 defines a supremum of the weak critical norm over all
solutions of bounded energy on intervals of bounded length;
Theorem 6.1 makes this supremum finite. Proposition 3.3(ii)
bounds the strong restriction norm by a function of the
conserved energy and the weak critical norm. Since both are
bounded in terms of $R$, the resulting strong-norm bound
depends only on $R$.

The unforced solution is continuous in $H^1$
on a neighborhood of the unit interval. Given an atomic extension
of its restriction to the open interval, truncate that extension
to the interval and attach the free evolutions of its two endpoint
traces, with a continuous cutoff before the left endpoint. The
construction in Lemma~\ref{det:concatenation} gives a continuous
whole-line extension with an absolute multiple of the original
norm. The endpoint norms are bounded by the same norm by taking
one-sided limits. This transfers the cited open-interval bound
to the restriction convention used here. Time translation leaves
all these bounds unchanged. Taking the supremum of the
resulting bounds over initial norms at most $R$ gives a finite
nondecreasing choice of $B_0(R)$. Mass and energy conservation
are part of the cited global theorem. This completes the proof.
\end{proof}

For intervals longer than one, the convention in \cite{IP2012}
takes a supremum over restriction norms on subintervals of length
at most one. We use the whole-interval extension norm
\eqref{det:Xrestriction}, as in \cite[(2.4)]{LOT2025}.
Longer-interval estimates follow from \eqref{det:unforcedbound}
and Lemma~\ref{det:concatenation}, with constants that may depend
on the interval length.

\begin{proposition}\label{det:nonlinearinput}
For $\theta=16/369$ and every interval $I$ with $|I|\leq1$,
the following estimates hold with implicit constants independent of $I$:
\begin{align}
 \|\mathcal N(v)\|_{N^1(I)}
 &\lesssim\|v\|_{X^1(I)}\|v\|_{M(I)}^4
  =\|v\|_{X^1(I)}^3\|v\|_{Q(I)}^2,
 \label{det:pure}\\
 \|E(v,f)\|_{N^1(I)}
 &\lesssim |I|^\theta
 \left(\|v\|_{X^1(I)}^2\|v\|_{M(I)}^2\|f\|_{Z(I)}
 +\sum_{k=2}^5\|v\|_{X^1(I)}^{5-k}\|f\|_{Z(I)}^k\right),
 \label{det:mixed}
\end{align}
where $E(v,f)=\mathcal N(v+f)-\mathcal N(v)$.
The first estimate holds for $v\in X^1(I)$; the second initially
holds for such $v$ and smooth $f$, and extends to all $f\in Z(I)$.
\end{proposition}

\begin{proof}
The pure estimate is the diagonal form of
\cite[Lemma 3.2]{IP2012}. The mixed estimate is the part containing
at least one perturbation factor in the proof of
\cite[Proposition 3.1]{LOT2025}; the same proposition includes
the pure term. We record explicitly the time gain in the mixed
estimate. For smooth $f$, put
\[
 \|f\|_{\widetilde Z(I)}^2
  =\sum_{N\geq1}N^2\|P_Nf\|_{L^{18}(I\times\mathbb T^3)}^2,
 \qquad \theta_0=\frac1{18}-\frac1{82}=\frac{16}{369}.
\]
H\"older's inequality and the normalization of spatial measure give
\[
 \|f\|_{\widetilde Z(I)}\leq |I|^{\theta_0}\|f\|_{Z(I)}.
\]
Write $E(v,f)=\sum_{k=1}^5E_k(v,f)$, where $E_k$ collects the
monomials with exactly $k$ factors equal to $f$ or $\overline f$.
The four mixed cases in the proof of
\cite[Proposition 3.1, in particular (3.6)]{LOT2025} give
\begin{align*}
 \|E_1(v,f)\|_{N^1(I)}
 &\lesssim\|v\|_{X^1(I)}^2\|v\|_{M(I)}^2
                 \|f\|_{\widetilde Z(I)},\\
 \|E_k(v,f)\|_{N^1(I)}
 &\lesssim\|v\|_{X^1(I)}^{5-k}
                 \|f\|_{\widetilde Z(I)}^k,
 \qquad 2\leq k\leq4.
\end{align*}
The square-function estimate and Minkowski's inequality imply
\[
 \|f\|_{L^{82}(I;W^{1,82})}\lesssim\|f\|_{Z(I)}.
\]
The embeddings $W^{1,82}\hookrightarrow L^\infty$ and
$W^{1,82}\hookrightarrow H^1$ imply
\[
 \|\mathcal N(f)(t)\|_{H^1}
 \lesssim\|f(t)\|_{L^\infty}^4\|f(t)\|_{H^1}
 \lesssim\|f(t)\|_{W^{1,82}}^5,
\]
where the gradient term follows from
$|\nabla(|f|^4f)|\leq5|f|^4|\nabla f|$.
Thus \eqref{det:L1source} and H\"older's inequality yield
\begin{align*}
 \|E_5(v,f)\|_{N^1(I)}
 &=\|\mathcal N(f)\|_{N^1(I)}\\
 &\lesssim\|f\|_{L^5(I;H^1)}
           \|f\|_{L^5(I;L^\infty)}^4
 \lesssim |I|^{77/82}\|f\|_{Z(I)}^5.
\end{align*}
Since $|I|\leq1$, $k\theta_0\geq\theta_0$ for $k\geq1$, and
$77/82>\theta_0$, these bounds prove \eqref{det:mixed}
with $\theta=\theta_0$ for smooth $f$.
The extension to general $f\in Z(I)$ follows from the density
and difference estimates in Lemmas~\ref{det:H3toZ} and
\ref{det:differencebounds}; Lemma~\ref{det:measurable}
identifies the limit with the $H^{-1}$-valued Duhamel integral.
The proofs of these lemmas use only the estimates for smooth
$f$ established above. This completes the proof.
\end{proof}

\begin{lemma}
\label{det:H3toZ}
Let $I$ be finite. If $f\in L^{82}(I;H^3)$ is strongly
measurable, then
\begin{equation}\label{det:H3general}
 \|f\|_{Z(I)}\lesssim\|f\|_{L^{82}(I;H^3)}.
\end{equation}
In particular, if $f\in L^\infty(I;H^3)$, then
\begin{equation}\label{det:H3bound}
 \|f\|_{Z(I)}\lesssim |I|^{1/82}
                   \|f\|_{L^\infty(I;H^3)}.
\end{equation}
The same bound holds for the difference of two such functions.
Furthermore,
\begin{equation}\label{det:ZtoH1}
 \|f\|_{L^{82}(I;H^1)}\lesssim\|f\|_{Z(I)}.
\end{equation}
Smooth functions with finite spatial Fourier support are dense
in $Z(I)$.
\end{lemma}

\begin{proof}
Bernstein's inequality gives, for almost every $t$,
\begin{align*}
 N\|P_Nf(t)\|_{L^{82}_x}
 &\lesssim N^{1+3(1/2-1/82)}\|P_Nf(t)\|_2\\
 &\lesssim N^{-1/2-3/82}\|f(t)\|_{H^3}.
\end{align*}
Integrate in time, square, and sum over $N$ to obtain
\[
 \|f\|_{Z(I)}^2\lesssim
 \|f\|_{L^{82}(I;H^3)}^2
 \sum_{N\geq1}N^{-1-6/82}.
\]
The dyadic series converges, proving \eqref{det:H3general};
H\"older's inequality gives \eqref{det:H3bound}.

The Littlewood--Paley characterization of $H^1$, Minkowski's
inequality in $L^{41}(I)$, and the finite volume embedding
$L^{82}_x\subset L^2_x$ give
\[
 \|f\|_{L^{82}_tH^1_x}
 \lesssim\left(\sum_NN^2
              \|P_Nf\|_{L^{82}_tL^2_x}^2\right)^{1/2}
 \lesssim\|f\|_{Z(I)}.
\]
To prove density, let $L_M$ be the Fourier multiplier with
symbol $\rho(\xi/M)$, where $M$ tends to infinity through
dyadic integers. The multipliers $L_M$ are uniformly bounded
on $L^{82}_x$, commute with $P_N$, and satisfy
$(\mathrm{Id}-L_M)P_N=0$ for $2N\leq M$. It follows that
\[
 \|f-L_Mf\|_{Z(I)}^2
 \lesssim\sum_{N>M/2}N^2
            \|P_Nf\|_{L^{82}(I\times\mathbb T^3)}^2
 \longrightarrow0.
\]
For fixed $M$, write
$L_Mf=\sum_{|\xi|<2M}c_\xi(t)e^{i\xi\cdot x}$.
Each $c_\xi$ belongs to $L^{82}(I)$. Choose smooth
$c_{\xi,n}$ converging to $c_\xi$ in that space. Since there
are finitely many frequencies,
\[
 \left\|L_Mf-\sum_{|\xi|<2M}c_{\xi,n}(t)e^{i\xi\cdot x}
       \right\|_{Z(I)}
 \lesssim_M\sum_{|\xi|<2M}\|c_\xi-c_{\xi,n}\|_{L^{82}(I)}
 \longrightarrow0.
\]
Choose $M$ and then $n$ to obtain density in $Z(I)$.
If $f\in L^\infty(I;H^3)$, extend it by zero in time.
Choose $\varphi\in C_c^\infty(\mathbb R)$ with $\varphi\geq0$
and $\int_{\mathbb R}\varphi(t)\,dt=1$, and set
$\varphi_\varepsilon(t)=\varepsilon^{-1}\varphi(t/\varepsilon)$.
The bound $|\rho|\leq1$ and the triangle inequality in $H^3$ give
\[
 \|\varphi_\varepsilon*_t L_Mf\|_{L^\infty(\mathbb R;H^3)}
 \leq\|f\|_{L^\infty(I;H^3)}.
\]
For fixed $M$, time convolution converges in each of the
finitely many $L^{82}$ coefficient norms as
$\varepsilon\downarrow0$, proving the last assertion.
This completes the proof.
\end{proof}

Fix $C_Z\geq1$ so that \eqref{det:H3bound} holds with $C_Z$
as its implicit constant in the restart and stability estimates.

\subsection{Polynomial difference estimates and measurable forcing}
\label{det:differences}

The polynomial structure allows us to derive the required
difference estimates from the quintic bounds. We first record
a Banach-valued polynomial estimate.

\begin{lemma}
\label{det:polynomial}
For each integer $m\geq1$, every
polynomial $P:\mathbb R\to\mathcal B$ of degree at most $m$,
with values in a real Banach space $\mathcal B$, satisfies
\begin{equation}\label{det:polynomialbound}
 \|P'(0)\|_{\mathcal B}
 \lesssim_m r^{-1}\sup_{|t|\leq r}\|P(t)\|_{\mathcal B},
 \qquad r>0.
\end{equation}
Complex Banach spaces may be regarded as real ones in this
statement.
\end{lemma}

\begin{proof}
First take $r=1$. Choose the nodes $t_j=-1+2j/m$,
$0\leq j\leq m$, and define their scalar Lagrange polynomials by
\[
 \ell_j(t)=\prod_{\substack{0\leq k\leq m\\k\ne j}}
             \frac{t-t_k}{t_j-t_k},\qquad 0\leq j\leq m.
\]
The identity
$P(t)=\sum_{j=0}^mP(t_j)\ell_j(t)$ holds coefficientwise in
$\mathcal B$. Therefore
\[
 \|P'(0)\|_{\mathcal B}
 \leq\sum_{j=0}^m|\ell_j'(0)|\,\|P(t_j)\|_{\mathcal B}
 \lesssim_m\sup_{|t|\leq1}\|P(t)\|_{\mathcal B}.
\]
For general $r>0$, apply this inequality to $Q(t)=P(rt)$
and use $Q'(0)=rP'(0)$. This completes the proof.
\end{proof}

\begin{lemma}\label{det:differencebounds}
Let $|I|\leq1$. There are fixed constants
$C_{\mathrm{nl}},C_{\mathrm{mix}}\geq1$ with the following properties.
\begin{enumerate}[label=\textup{(\roman*)}]
\item If $B\geq1$, $0<\eta\leq1$, and $v,w\in X^1(I)$ satisfy
\[
 \|v\|_{X^1(I)}+\|w\|_{X^1(I)}\leq B,
 \qquad \|v\|_{Q(I)}+\|w\|_{Q(I)}\leq\eta,
\]
then
\begin{equation}\label{det:puredifference}
 \|\mathcal N(v)-\mathcal N(w)\|_{N^1(I)}
 \leq C_{\mathrm{nl}}B^3\eta\|v-w\|_{X^1(I)}.
\end{equation}
\item If $B,K\geq0$, $v,w\in X^1(I)$, and $f,g\in Z(I)$ satisfy
\[
 \|v\|_{X^1(I)}+\|w\|_{X^1(I)}\leq B,
 \qquad\|f\|_{Z(I)}+\|g\|_{Z(I)}\leq K,
\]
then, for $P(B,K)=C_{\mathrm{mix}}(1+B+K)^5$,
\begin{align}
 \|E(v,f)\|_{N^1(I)}&\leq |I|^\theta P(B,K),
 \label{det:coarsemixed}\\
 \|E(v,f)-E(w,g)\|_{N^1(I)}
 &\leq |I|^\theta P(B,K)
       \bigl(\|v-w\|_{X^1(I)}+\|f-g\|_{Z(I)}\bigr).
 \label{det:mixeddifference}
\end{align}
The constants and the exponent $\theta$ are independent of $I$.
\end{enumerate}
\end{lemma}

\begin{proof}
For the first assertion, put $h=v-w$ and
$z_\sigma=(1-\sigma)w+\sigma v$, $0\leq\sigma\leq1$.
If $h=0$, the two sides of \eqref{det:puredifference} vanish.
Assume $h\ne0$. The segment satisfies
$\|z_\sigma\|_{X^1(I)}\leq B$ and
$\|z_\sigma\|_{Q(I)}\leq\eta$. Take
\[
 r=\frac{\eta}{(1+C_Q)\|h\|_{X^1(I)}}.
\]
For $|t|\leq r$ we have $\|z_\sigma+th\|_{X^1(I)}\leq2B$ and
$\|z_\sigma+th\|_{Q(I)}\leq2\eta$. By \eqref{det:pure},
\[
 \sup_{|t|\leq r}
 \|\mathcal I_a\mathcal N(z_\sigma+th)\|_{X^1(I)}
 \lesssim B^3\eta^2.
\]
The expression inside the norm is a real polynomial in $t$ of
degree at most five. Indeed, $\mathcal N(z)=z^3\overline z^{2}$,
and $\mathcal I_a$ is real linear. The real directional derivative is
\[
 D\mathcal N(z)[h]=3|z|^4h+2z^3\overline z\,\overline h.
\]
Lemma~\ref{det:polynomial} therefore yields
\[
 \|\mathcal I_aD\mathcal N(z_\sigma)[h]\|_{X^1(I)}
 \lesssim B^3\eta\|h\|_{X^1(I)}.
\]
The fundamental theorem of calculus for this Banach-valued
polynomial gives
\[
 \mathcal I_a\bigl(\mathcal N(v)-\mathcal N(w)\bigr)
 =\int_0^1\mathcal I_aD\mathcal N(z_\sigma)[h]\,d\sigma.
\]
Taking the $X^1(I)$ norm and integrating the preceding bound
proves \eqref{det:puredifference}.

For the second assertion, \eqref{det:Qembed} gives
$\|v\|_{M(I)}\lesssim\|v\|_{X^1(I)}$. Consequently
\eqref{det:mixed} implies \eqref{det:coarsemixed}.
First suppose that $f,g$ are smooth. Let $h=v-w$, $k=f-g$, and
consider the segment $z_\sigma=(1-\sigma)w+\sigma v$,
$q_\sigma=(1-\sigma)g+\sigma f$.
If $d=\|h\|_{X^1(I)}+\|k\|_{Z(I)}>0$, apply
Lemma~\ref{det:polynomial} with $r=d^{-1}$ to
\[
 t\longmapsto\mathcal I_aE(z_\sigma+th,q_\sigma+tk).
\]
Its degree is at most five, and on $[-r,r]$ its two arguments
have norms at most $B+1$ and $K+1$. The coarse mixed estimate gives
\[
 \sup_{|t|\leq r}
 \|\mathcal I_aE(z_\sigma+th,q_\sigma+tk)\|_{X^1(I)}
 \lesssim |I|^\theta(1+B+K)^5.
\]
Its derivative at zero is
\[
 \mathcal I_a\bigl(D\mathcal N(z_\sigma+q_\sigma)[h+k]
                         -D\mathcal N(z_\sigma)[h]\bigr),
\]
whose $X^1(I)$ norm is bounded by $|I|^\theta P(B,K)d$
after fixing $C_{\mathrm{mix}}$ sufficiently large. Integrating
this derivative over $0\leq\sigma\leq1$ yields
\[
 \mathcal I_a\bigl(E(v,f)-E(w,g)\bigr)
 =\int_0^1\mathcal I_a\bigl(D\mathcal N(z_\sigma+q_\sigma)[h+k]
                      -D\mathcal N(z_\sigma)[h]\bigr)\,d\sigma,
\]
and hence \eqref{det:mixeddifference}. For $d=0$, both
arguments agree and the difference vanishes.

Finally, approximate arbitrary $f,g\in Z(I)$ by the smooth
functions from Lemma~\ref{det:H3toZ}. The estimate just proved
makes the corresponding mixed Duhamel terms Cauchy in $X^1(I)$.
Their limits give \eqref{det:coarsemixed} and
\eqref{det:mixeddifference} for all $f,g\in Z(I)$.
Lemma~\ref{det:measurable} identifies these limits with the
integral in \eqref{det:source}. This completes the proof.
\end{proof}

\begin{lemma}
\label{det:measurable}
For $v\in X^1(I)$ and strongly measurable $f\in Z(I)$,
$\mathcal N(v+f)$ belongs to $L^1(I;H^{-1})$. The Duhamel term
defined in this space agrees with the $X^1(I)$ representative
provided by Proposition~\ref{det:nonlinearinput} and
Lemma~\ref{det:differencebounds}. In particular it is continuous
in $H^1$ and vanishes at the left endpoint.
\end{lemma}

\begin{proof}
The scalar inequality
\[
 |\mathcal N(z)-\mathcal N(\widetilde z)|
 \lesssim (|z|^4+|\widetilde z|^4)|z-\widetilde z|
\]
and H\"older's inequality in space imply
\begin{align*}
 \|\mathcal N(u)-\mathcal N(\widetilde u)\|_{L^{6/5}}
 &\lesssim\bigl(\|u\|_6^4+\|\widetilde u\|_6^4\bigr)
                     \|u-\widetilde u\|_6.
\end{align*}
Using $H^1\hookrightarrow L^6$ and
$L^{6/5}\hookrightarrow H^{-1}$ gives
\begin{equation}\label{det:Hminusone}
 \|\mathcal N(u)-\mathcal N(\widetilde u)\|_{H^{-1}}
 \lesssim\bigl(\|u\|_{H^1}^4+\|\widetilde u\|_{H^1}^4\bigr)
               \|u-\widetilde u\|_{H^1}.
\end{equation}
The map $\mathcal N:H^1\to H^{-1}$ is continuous, so its
composition with the strongly measurable function $v+f$ is
strongly measurable. Moreover,
\[
 \|v+f\|_{L^{82}(I;H^1)}
 \leq |I|^{1/82}\|v\|_{C(I;H^1)}
                      +\|f\|_{L^{82}(I;H^1)}<\infty,
\]
and therefore
\[
 \|\mathcal N(v+f)\|_{L^1(I;H^{-1})}
 \lesssim |I|^{77/82}\|v+f\|_{L^{82}(I;H^1)}^5.
\]
Let $f_n\to f$ in $Z(I)$ be the smooth approximations of
Lemma~\ref{det:H3toZ}, and put $u_n=v+f_n$, $u=v+f$.
Then $u_n\to u$ in $L^{82}(I;H^1)$, and
\begin{align*}
 &\|\mathcal N(u_n)-\mathcal N(u)\|_{L^1(I;H^{-1})}\\
 &\quad\lesssim |I|^{77/82}
  \bigl(\|u_n\|_{L^{82}(I;H^1)}^4+
               \|u\|_{L^{82}(I;H^1)}^4\bigr)
                         \|f_n-f\|_{L^{82}(I;H^1)}
 \longrightarrow0.
\end{align*}
Unitarity of $S(t)$ on $H^{-1}$ gives
\[
 \sup_{a\leq t\leq b}
 \left\|\int_a^t S(t-s)
           [\mathcal N(u_n(s))-\mathcal N(u(s))]\,ds\right\|_{H^{-1}}
 \leq\|\mathcal N(u_n)-\mathcal N(u)\|_{L^1(I;H^{-1})}.
\]
For $|I|\leq1$, Lemma~\ref{det:differencebounds} also makes
these Duhamel terms Cauchy in $X^1(I)$. The embeddings
$X^1(I)\hookrightarrow C(I;H^1)\hookrightarrow C(I;H^{-1})$
identify the two limits. Their value at $a$ is zero.
For a longer interval, apply this argument on finitely many
pieces of length at most one, propagate the endpoint values
by \eqref{det:linear}, and concatenate using
Lemma~\ref{det:concatenation}. This completes the proof.
\end{proof}

\subsection{Comparison on a short interval}
\label{det:comparisonsection}

We record first the comparison estimate that will be used both in
the construction and in the approximation argument. All constants in
this subsection refer to the fixed torus and the fixed dyadic
decomposition. Choose $C_{\mathrm{mix}}$ large enough for both mixed
estimates in Lemma~\ref{det:differencebounds}, and keep
$C_{\mathrm{nl}}$, $C_{\mathrm{mix}}$, and the polynomial $P$ fixed
in the following parameter choices.

\begin{lemma}\label{det:comparison}
Let $I=[a,b]$, $|I|\leq1$, and let $v,\widetilde v\in X^1(I)$
solve \eqref{det:equation} with perturbations $f,g\in Z(I)$,
respectively. Suppose that, for $B\geq1$, $0<\eta\leq1$, and
$K\geq0$,
\begin{gather*}
 \|v\|_{X^1(I)}+\|\widetilde v\|_{X^1(I)}\leq B,
 \qquad \|v\|_{Q(I)}+\|\widetilde v\|_{Q(I)}\leq\eta,\\
 \|f\|_{Z(I)}+\|g\|_{Z(I)}\leq K.
\end{gather*}
If
\begin{equation}\label{det:comparisonconditions}
 C_{\mathrm{nl}}B^3\eta\leq\tfrac14,
 \qquad |I|^\theta P(B,K)\leq\tfrac14,
\end{equation}
then
\begin{align}
 \|v-\widetilde v\|_{X^1(I)}
 &\leq2\|v(a)-\widetilde v(a)\|_{H^1}\notag\\
 &\quad+2|I|^\theta P(B,K)\|f-g\|_{Z(I)}.
 \label{det:comparisonbound}
\end{align}
\end{lemma}

\begin{proof}
Write $d=v-\widetilde v$. Subtracting the two integral equations
and using $\mathcal N(v+f)=\mathcal N(v)+E(v,f)$ gives
\begin{align*}
 d(t)={}&S(t-a)d(a)
 -i\int_a^tS(t-s)\big[\mathcal N(v)-\mathcal N(\widetilde v)\big](s)\,ds\\
 &-i\int_a^tS(t-s)\big[E(v,f)-E(\widetilde v,g)\big](s)\,ds.
\end{align*}
The linear estimate and Lemma~\ref{det:differencebounds} yield
\begin{align*}
 \|d\|_{X^1(I)}
 \leq{}&\|d(a)\|_{H^1}
 +\big[C_{\mathrm{nl}}B^3\eta+|I|^\theta P(B,K)\big]\|d\|_{X^1(I)}\\
 &+|I|^\theta P(B,K)\|f-g\|_{Z(I)}.
\end{align*}
By \eqref{det:comparisonconditions}, this yields
\[
 \tfrac12\|d\|_{X^1(I)}
 \leq\|d(a)\|_{H^1}
       +|I|^\theta P(B,K)\|f-g\|_{Z(I)},
\]
which is \eqref{det:comparisonbound}. This completes the proof.
\end{proof}

\begin{corollary}
\label{det:uniqueness}
For an interval $J$ and a prescribed perturbation
$f\in Z_{\mathrm{loc}}(J)$, two
solutions of \eqref{det:equation} with the same initial value
and belonging to $X^1_{\mathrm{loc}}(J)$ agree on their common
interval of existence in $J$.
\end{corollary}

\begin{proof}
Fix a compact common interval $I=[a,b]$ and partition it into
finitely many intervals of length at most one. Choose $B\geq1$
so that the sum of the two $X^1$ norms on each such interval
is at most $B$. Set
\begin{gather*}
 \eta=\min\{1,(4C_{\mathrm{nl}}B^3)^{-1}\},\qquad
 K=2\|f\|_{Z(I)},\\
 \ell=\min\{1,[4P(B,K)]^{-1/\theta}\}.
\end{gather*}
Apply Lemma~\ref{det:subdivision} to the two solutions on each
of these intervals with threshold $\eta/2$, and impose the
length bound $\ell$.
For the resulting partition $a=t_0<\cdots<t_J=b$, the
hypotheses of Lemma~\ref{det:comparison} hold on each piece.
Since the forcing difference vanishes, that lemma gives
\[
 \|v-\widetilde v\|_{X^1([t_j,t_{j+1}])}
 \leq2\|v(t_j)-\widetilde v(t_j)\|_{H^1}.
\]
The initial difference is zero. Evaluation at the right
endpoint and induction over $j=0,\ldots,J-1$ prove equality
on all of $I$. This completes the proof.
\end{proof}

\subsection{A uniform restart time on bounded sets}
\label{det:restartsection}

We compare the perturbed solution with the unforced nonlinear
flow. The unit-time bound gives a uniform upper bound for the
number of critical-norm subdivisions, while the subdivision
points may depend on the initial profile.

\begin{proposition}\label{det:restart}
For every $R\geq1$ and $K\geq0$ there are numbers
$\delta=\delta(R,K)\in(0,1]$ and $A=A(R,K)<\infty$ with the
following property. Given $a\in\mathbb R$, $v_a\in H^1$ with
$\|v_a\|_{H^1}\leq R$, and a strongly measurable perturbation
$f$ satisfying
\[
 \|f\|_{L^\infty([a,a+\delta];H^3)}\leq K,
\]
equation \eqref{det:equation} has a solution on
$[a,a+\delta]$, unique in the local $X^1$ class, and
\begin{equation}\label{det:restartbound}
 \|v\|_{X^1([a,a+\delta])}\leq A.
\end{equation}
The same assertion holds on any shorter interval. The constants
are independent of the Fourier profile of $v_a$ and of the
number, times, or sizes of the jumps of $f$, subject to the
displayed bound.
\end{proposition}

\begin{proof}
Let $w$ be the unforced solution with $w(a)=v_a$. Proposition
\ref{det:unforced} gives a number
$B=1+B_0(R)$ such that
\[
 \|w\|_{X^1([a,a+1])}\leq B.
\]
Choose $0<\eta\leq1$ and $r>0$ so that
\begin{equation}\label{det:restartsmall}
 C_{\mathrm{nl}}(2B+2)^3\eta\leq\tfrac18,
 \qquad 0<r\leq\min\{1,\eta/[8(1+C_Q)]\}.
\end{equation}
Here $C_Q$ is the embedding constant in
\eqref{det:Qembed}. Partition the entire unit interval so
that $\|w\|_{Q(I_j)}\leq\eta/8$ on every piece. By
Lemma~\ref{det:subdivision}, the number of pieces is bounded by
the integer
\begin{equation}\label{det:Jstar}
 J_*=
 \left\lceil 2+\sum_{p\in\{41/10,100\}}
           (16C_QB/\eta)^p\right\rceil.
\end{equation}
For every $0<\delta\leq1$, intersecting this partition with
$[a,a+\delta]$ leaves at most $J_*$ nonempty pieces.

By Lemma~\ref{det:H3toZ}, every such piece, of length at most
one, satisfies $\|f\|_{Z(I_j)}\leq C_ZK$. Choose a constant
$P_*\geq1$, depending only on $B$ and $K$, that dominates the
coarse mixed bound and the mixed difference bound whenever
the solution norms have sum at most $2B+2$ and the forcing
norms have sum at most $2C_ZK$. For example, a sufficiently large
fixed multiple of $(1+2B+2+2C_ZK)^5$ has this property. Choose
$\delta\in(0,1]$ so small that
\begin{equation}\label{det:deltachoice}
 \delta^\theta P_*\leq\tfrac18,
 \qquad 2^{J_*+4}\delta^\theta P_*\leq r.
\end{equation}
Set $h_* =\delta^\theta P_*$.

Write the restricted partition as
$a=a_0<a_1<\cdots<a_J=a+\delta$, where $J\leq J_*$.
We construct $d=v-w$ successively on
$I_j=[a_j,a_{j+1}]$, $j=0,\ldots,J-1$. Suppose its initial value $d_j$ is
known, and set $e_j=\|d_j\|_{H^1}$. Consider the closed
affine ball
\[
 \mathcal B_j=\{d\in X^1(I_j):d(a_j)=d_j,
                              \ \|d\|_{X^1(I_j)}\leq r\}.
\]
When $e_j\leq r$, the free evolution of $d_j$ belongs to this
ball, so it is nonempty. For $d\in\mathcal B_j$, define
\begin{align*}
 (\Phi_jd)(t)={}&S(t-a_j)d_j
 -i\int_{a_j}^t S(t-s)
       \big[\mathcal N(w+d)-\mathcal N(w)\big](s)\,ds\\
 &-i\int_{a_j}^t S(t-s)E(w+d,f)(s)\,ds.
\end{align*}
These Duhamel terms belong to $X^1(I_j)$, including for the
present measurable $f$, by Lemma~\ref{det:measurable}.
Every candidate satisfies
\[
 \|w+d\|_{X^1(I_j)}\leq B+1,
 \qquad
 \|w+d\|_{Q(I_j)}\leq\eta/8+C_Qr\leq\eta/4.
\]
For two candidates, or for a candidate and $w$, the sums of
the $X^1$ and $Q$ norms are therefore bounded by $2B+2$ and
$\eta$, respectively. Lemma~\ref{det:differencebounds} and
\eqref{det:restartsmall}--\eqref{det:deltachoice} give
\begin{align}
 \|\Phi_jd\|_{X^1(I_j)}
 &\leq e_j+\tfrac18\|d\|_{X^1(I_j)}+h_* ,
 \label{det:ballmap}\\
 \|\Phi_jd-\Phi_j\widetilde d\|_{X^1(I_j)}
 &\leq\tfrac14\|d-\widetilde d\|_{X^1(I_j)}.
 \label{det:ballcontraction}
\end{align}
The first estimate uses the coarse bound on $E(w+d,f)$;
the second uses its difference bound with the same $f$.
The initial value is preserved because each Duhamel term
vanishes at $a_j$.

If $e_j+h_*\leq r/2$, then
$\|\Phi_jd\|_{X^1(I_j)}\leq r/2+r/8=5r/8$ for
$d\in\mathcal B_j$. Thus \eqref{det:ballcontraction} makes
$\Phi_j$ a strict contraction of $\mathcal B_j$ into itself.
For its fixed point, \eqref{det:ballmap} gives
$\frac78\|d\|_{X^1(I_j)}\leq e_j+h_*$, and hence
\begin{equation}\label{det:errorrecurrence}
 \|d\|_{X^1(I_j)}\leq2e_j+2h_* ,
 \qquad e_{j+1}\leq2e_j+2h_* .
\end{equation}
Since $e_0=0$, induction yields
\[
 e_j\leq2(2^j-1)h_*.
\]
For $0\leq j<J_*$, \eqref{det:deltachoice} gives
\[
 e_j+h_*\leq(2^{j+1}-1)h_*
 \leq2^{J_*}h_*\leq r/16<r/2.
\]
Thus the fixed point can be constructed on every piece.

The solutions agree at common endpoints. The concatenation
estimate, Lemma~\ref{det:concatenation}, gives
\[
 \|v\|_{X^1([a,a+\delta])}
 \leq C_{\mathrm{con}}J_*(B+r)=:A.
\]
To verify the integral equation across a joining time $a_j$,
suppose the equation holds up to $a_j$. For $t\in I_j$,
the group identity $S(t-a_j)S(a_j-s)=S(t-s)$ gives
\begin{align*}
 v(t)
 &=S(t-a_j)\left[S(a_j-a)v_a
       -i\int_a^{a_j}S(a_j-s)\mathcal N(v(s)+f(s))\,ds\right]\\
 &\quad-i\int_{a_j}^{t}S(t-s)\mathcal N(v(s)+f(s))\,ds\\
 &=S(t-a)v_a-i\int_a^tS(t-s)\mathcal N(v(s)+f(s))\,ds.
\end{align*}
Induction proves \eqref{det:equation} on $[a,a+\delta]$.
Corollary~\ref{det:uniqueness} gives uniqueness. On any
shorter prescribed interval, extend $f$ by zero, apply the
construction, and restrict the solution. This completes the proof.
\end{proof}

\subsection{Continuation and stability on bounded sets}
\label{det:globalstabilitysection}

\begin{proposition}
\label{det:continuation}
Let $T>0$, let $f\in L^\infty([0,T];H^3)$ with norm at most
$K$, and let $v$ be a local $X^1$ solution on $[0,T_*)$,
where $0<T_*\leq T$. If
\begin{equation}\label{det:aprioriH1}
 \sup_{0\leq t<T_*}\|v(t)\|_{H^1}\leq R,
 \qquad R\geq1,
\end{equation}
then $v$ has a continuous $H^1$ trace at $T_*$. If $T_*<T$,
it extends to a strictly larger interval. Moreover, every
solution on $[0,T]$ satisfying the same $H^1$ bound obeys
\begin{equation}\label{det:AT}
 \|v\|_{X^1([0,T])}\leq A_T(R,K),
\end{equation}
where $A_T(R,K)<\infty$ depends only on the displayed
parameters.
\end{proposition}

\begin{proof}
Use $\delta(R,K)$ from Proposition~\ref{det:restart} and choose
$a<T_*$ with $T_*-a<\delta/2$. Restart from the already
defined value $v(a)$. The restarted solution exists up to
$\min\{a+\delta,T\}$ and coincides with the old solution on
their overlap by uniqueness. If $\widetilde v$ denotes this
restarted solution, then $\min\{a+\delta,T\}\geq T_*$ and
\[
 \lim_{t\uparrow T_*}\|v(t)-\widetilde v(T_*)\|_{H^1}
 =\lim_{t\uparrow T_*}\|\widetilde v(t)-\widetilde v(T_*)\|_{H^1}=0.
\]
This defines the trace at $T_*$. If $T_*<T$, then both
$a+\delta$ and $T$ are strictly greater than $T_*$, so the
same solution extends beyond $T_*$.

For the whole-interval estimate, subdivide $[0,T]$ into at most
$n_T=1+\lceil T/\delta\rceil$ intervals of length at most
$\delta$. At each left endpoint the initial norm is at most
$R$. Uniqueness identifies the given solution with the
solution of Proposition~\ref{det:restart}, whose norm on that
piece is at most $A(R,K)$. Concatenation gives
\[
 A_T(R,K)=\max\{1,C_{\mathrm{con}}n_TA(R,K)\}
\]
as one possible choice. This completes the proof.
\end{proof}

\begin{theorem}\label{det:stability}
Fix $T>0$, $R\geq1$, and $K\geq0$. Let $v,\widetilde v$ be
local $X^1$ solutions on $[0,T]$ with perturbations $f,g$,
respectively. Assume
\begin{gather*}
 \sup_{0\leq t\leq T}\|v(t)\|_{H^1}\leq R,
 \qquad
 \sup_{0\leq t\leq T}\|\widetilde v(t)\|_{H^1}\leq R,\\
 \|f\|_{L^\infty([0,T];H^3)}\leq K,
 \qquad
 \|g\|_{L^\infty([0,T];H^3)}\leq K.
\end{gather*}
There is a finite constant $C_T(R,K)$ such that
\begin{align}
 \|v-\widetilde v\|_{X^1([0,T])}
 &\leq C_T(R,K)
 \big(\|v(0)-\widetilde v(0)\|_{H^1}
                         +\|f-g\|_{Z([0,T])}\big),
 \label{det:stabilityZ}\\
 \|v-\widetilde v\|_{X^1([0,T])}
 &\leq C_T(R,K)
 \big(\|v(0)-\widetilde v(0)\|_{H^1}
                  +\|f-g\|_{L^\infty([0,T];H^3)}\big).
 \label{det:stabilityH3}
\end{align}
The constant is uniform over the indicated bounded sets.
\end{theorem}

\begin{proof}
Proposition~\ref{det:continuation} bounds each whole-interval
$X^1$ norm by $A_T=A_T(R,K)$. Set $B=2A_T+1$, and choose
$0<\eta\leq1$ so that $C_{\mathrm{nl}}B^3\eta\leq1/4$. On every interval
of length at most one, Lemma~\ref{det:H3toZ} bounds the sum of
the two forcing norms by $K_1=2C_ZK$. Choose $0<\ell\leq1$
so that
\[
 \ell^\theta P(B,K_1)\leq\tfrac14.
\]

We next bound the number of intervals needed to apply the
short-interval comparison simultaneously to both solutions.
Divide $[0,T]$ first into at most $n_1=1+\lceil T\rceil$
unit intervals. For either $z=v$ or $z=\widetilde v$ and
$p\in\{41/10,100\}$, additivity of the measure defining
$Q_p^p$ and the unit-interval embedding give
\[
 \|z\|_{Q_p([0,T])}^p\leq n_1(C_QA_T)^p.
\]
Repeat the simultaneous nonatomic-measure construction in the proof
of Lemma~\ref{det:subdivision} for these four finite measures,
with the requirement $\|z\|_{Q_p(I_j)}\leq\eta/4$ for each one.
Refine also by intervals of length at most $\ell$. The
resulting partition
$0=t_0<t_1<\cdots<t_J=T$ has
\begin{equation}\label{det:JT}
 J\leq J_T:=
 \left\lceil 2+T/\ell
       +2n_1\sum_{p\in\{41/10,100\}}
                   (4C_QA_T/\eta)^p\right\rceil.
\end{equation}
On every piece the sum of the two $Q$ norms is at most
$\eta$; all hypotheses of Lemma~\ref{det:comparison} hold.

Set $e_j=\|v(t_j)-\widetilde v(t_j)\|_{H^1}$ and
$D=\|f-g\|_{Z([0,T])}$. Restriction decreases the $Z$ norm,
so \eqref{det:comparisonbound} gives, with
$C_*=2P(B,K_1)$,
\begin{align*}
 \|v-\widetilde v\|_{X^1([t_j,t_{j+1}])}
 &\leq2e_j+C_*D,\\
 e_{j+1}&\leq2e_j+C_*D.
\end{align*}
Consequently,
\[
 e_j\leq2^j e_0+(2^j-1)C_*D.
\]
Substitution into the local bound gives, for $0\leq j<J$,
\[
 \|v-\widetilde v\|_{X^1([t_j,t_{j+1}])}
 \leq2^{j+1}e_0+(2^{j+1}-1)C_*D
 \leq2^{J_T}(e_0+C_*D).
\]
Lemma~\ref{det:concatenation} therefore yields
\begin{align*}
 \|v-\widetilde v\|_{X^1([0,T])}
 &\leq C_{\mathrm{con}}\sum_{j=0}^{J-1}
       \|v-\widetilde v\|_{X^1([t_j,t_{j+1}])}\\
 &\leq C_{\mathrm{con}}J_T2^{J_T}(e_0+C_*D).
\end{align*}
This is \eqref{det:stabilityZ}, since $J_T$ and $C_*$ depend
only on $T,R,K$. Finally,
\[
 D\leq C_ZT^{1/82}
             \|f-g\|_{L^\infty([0,T];H^3)}
\]
by \eqref{det:H3bound}. Enlarging the constant by
$1+C_ZT^{1/82}$ gives \eqref{det:stabilityH3}.
This completes the proof.
\end{proof}

\begin{corollary}\label{det:totalfield}
Under the assumptions of Theorem~\ref{det:stability}, put
$u=v+f$ and $\widetilde u=\widetilde v+g$. Then
\begin{align*}
 &\|u-\widetilde u\|_{L^\infty([0,T];H^1)}
       +\|v-\widetilde v\|_{X^1([0,T])}\\
 &\qquad\leq C_T(R,K)
 \big(\|v(0)-\widetilde v(0)\|_{H^1}
               +\|f-g\|_{L^\infty([0,T];H^3)}\big).
\end{align*}
For c\`adl\`ag forcing paths, the corresponding estimate with
pointwise suprema holds when the forcing bounds and the
forcing distance are also interpreted as pointwise suprema.
In that case $u$ is c\`adl\`ag in $H^1$ and
\[
 u(t)-u(t-)=f(t)-f(t-),\qquad 0<t\leq T.
\]
\end{corollary}

\begin{proof}
The identity $u-\widetilde u=(v-\widetilde v)+(f-g)$ gives
\begin{align*}
 \|u-\widetilde u\|_{L^\infty([0,T];H^1)}
 &\leq\|v-\widetilde v\|_{C([0,T];H^1)}
          +\|f-g\|_{L^\infty([0,T];H^1)}\\
 &\leq\|v-\widetilde v\|_{X^1([0,T])}
          +\|f-g\|_{L^\infty([0,T];H^3)}.
\end{align*}
Adding $\|v-\widetilde v\|_{X^1([0,T])}$ and applying
\eqref{det:stabilityH3} proves the asserted estimate after
enlarging $C_T(R,K)$. If $f$ and $g$ are c\`adl\`ag, the
same pointwise inequality applies at every $t\in[0,T]$,
including $T$, and gives the pointwise-supremum estimate.
The continuity of $v$ and existence of the left limits of $f$
show that $u$ is c\`adl\`ag and that
\[
 u(t)-u(t-)=v(t)-v(t-)+f(t)-f(t-)=f(t)-f(t-).
\]
This completes the proof.
\end{proof}

For $h\in H^3$ and $s,\widetilde s\in[0,T]$,
\[
 \|h\mathbf1_{[s,T]}-h\mathbf1_{[\widetilde s,T]}
            \|_{L^{82}([0,T];H^3)}
 =|s-\widetilde s|^{1/82}\|h\|_{H^3}.
\]
Consequently, for a fixed finite number of jumps with fixed
$H^3$ amplitudes, convergence of the jump times gives
convergence of the forcing paths in $L^{82}([0,T];H^3)$.
Lemma~\ref{det:H3toZ} and \eqref{det:stabilityZ} then give
convergence of the corresponding continuous remainders in
$X^1([0,T])$ under the bounds of Theorem~\ref{det:stability}.

\section{Stochastic estimates and construction on the prescribed basis}
\label{sec:probability}

We work throughout on the stochastic basis and with the Poisson random measure
specified in Theorem~\ref{thm:main}. All approximations use this same random
measure. For \(0<\varepsilon<1\), write
\[
 Z_\varepsilon=\{z\in Z_0:|z|>\varepsilon\},\qquad
 \widetilde N(\mathrm dt,\mathrm dz)
   =N(\mathrm dt,\mathrm dz)-\mathrm dt\,\nu(\mathrm dz).
\]
For \(j\in\{0,1,3\}\), we use the notation
\[
 q_j=\int_{Z_0}\|Gz\|_{H^j}^2\nu(\mathrm dz),\qquad
 q_3(r)=\int_{\{z\in Z_0:|z|\le r\}}\|Gz\|_{H^3}^2\nu(\mathrm dz),
 \qquad 0<r\leq1.
\]
The boundedness of \(G:\mathbb R^m\to H^3\) and the second-moment assumption
imply \(q_j<\infty\) and \(q_3(r)\to0\) as \(r\downarrow0\).
Constants below may depend on the fixed torus and on the indicated norms
of \(G\) and are uniform in the small-jump cutoff.

\subsection{The Poisson convolution and its spatial forcing norm}

Define the truncated stochastic convolution by
\begin{equation}\label{prob:truncatedconvolution}
 \Psi^\varepsilon(t)=
 \int_0^t\int_{Z_\varepsilon}S(t-s)Gz\,
       \widetilde N(\mathrm ds,\mathrm dz).
\end{equation}
The corresponding integral over \(Z_0\) will be denoted by \(\Psi\).
It is constructed by Hilbert-space stochastic integration.

\begin{lemma}
\label{prob:convolution}
The process \(\Psi\) has an adapted \(H^3\)-valued c\`adl\`ag version.
For every \(T<\infty\) and \(j\in\{0,1,3\}\),
\begin{equation}\label{prob:convolutionmax}
 \mathbb E\sup_{0\leq t\leq T}\|\Psi(t)\|_{H^j}^2
 \le 4Tq_j.
\end{equation}
The same bound holds with \(\Psi^\varepsilon\) in place of \(\Psi\).
Moreover,
\begin{equation}\label{prob:convolutiontail}
 \mathbb E\sup_{0\leq t\leq T}
   \|\Psi(t)-\Psi^\varepsilon(t)\|_{H^3}^2
 \le4Tq_3(\varepsilon),
\end{equation}
and, for \(0<\varepsilon,\eta<1\),
\begin{equation}\label{prob:convolutionannulus}
 \mathbb E\sup_{0\leq t\leq T}
   \|\Psi^\varepsilon(t)-\Psi^\eta(t)\|_{H^3}^2
 \le4Tq_3(\max\{\varepsilon,\eta\}).
\end{equation}
\end{lemma}

\begin{proof}
In the interaction representation, set
\[
 Y(t)=\int_0^t\int_{Z_0}S(-s)Gz\,\widetilde N(\mathrm ds,\mathrm dz).
\]
Its truncation is
\[
 Y^\varepsilon(t)=\int_0^t\int_{Z_\varepsilon}
                S(-s)Gz\,\widetilde N(\mathrm ds,\mathrm dz),
 \qquad \Psi^\varepsilon(t)=S(t)Y^\varepsilon(t).
\]
Since \(S(-s)\) is unitary on \(H^3\), its deterministic integrand satisfies
\[
 \int_0^T\int_{Z_0}\|S(-s)Gz\|_{H^3}^2
       \nu(\mathrm dz)\,\mathrm ds=Tq_3<\infty.
\]
The Hilbert-space compensated Poisson integral is therefore well defined
as a square-integrable martingale. We use the real Hilbert structure on
the complex Sobolev space when applying stochastic integration.
The Poisson isometry and the Hilbert-space Doob inequality give
\[
 \mathbb E\sup_{0\leq t\leq T}\|Y(t)\|_{H^j}^2
 \le4\mathbb E\|Y(T)\|_{H^j}^2
 =4\int_0^T\int_{Z_0}
       \|S(-s)Gz\|_{H^j}^2\nu(\mathrm dz)\,\mathrm ds
 =4Tq_j.
\]
The stochastic integral is the completion of predictable simple
integrands under the Poisson isometry; see
\cite{Applebaum2009,PeszatZabczyk2007}.

The identity \(\Psi(t)=S(t)Y(t)\) now defines the convolution.
The equality of the Sobolev norms of \(\Psi(t)\) and \(Y(t)\) transfers
the maximal bound to \(\Psi\). A c\`adl\`ag version is obtained from
the truncations as follows.
The truncated martingales \(Y^\varepsilon\) are c\`adl\`ag, and Doob's
inequality makes them Cauchy in the mean square of the uniform \(H^3\)
distance on every bounded interval. A subsequence converges uniformly
almost surely on all compact time intervals. Its limit is c\`adl\`ag.
Multiplication by the strongly continuous group \(S(t)\) preserves this
path property. For instance, right continuity at \(t\) follows from
\[
 \|S(t+h)Y(t+h)-S(t)Y(t)\|_{H^3}
 \le\|Y(t+h)-Y(t)\|_{H^3}
   +\|(S(h)-\mathrm{Id})Y(t)\|_{H^3}.
\]
For \(t>0\), the left limit is \(S(t)Y(t-)\), since
\begin{align*}
 &\|S(t-h)Y(t-h)-S(t)Y(t-)\|_{H^3}\\
 &\quad\leq\|Y(t-h)-Y(t-)\|_{H^3}
     +\|(S(-h)-\mathrm{Id})Y(t-)\|_{H^3}
 \longrightarrow0
\end{align*}
as \(h\downarrow0\).

The tail difference has interaction representation
\[
 S(-t)(\Psi(t)-\Psi^\varepsilon(t))
 =\int_0^t\int_{\{|z|\leq\varepsilon\}}
          S(-s)Gz\,\widetilde N(\mathrm ds,\mathrm dz).
\]
Doob's inequality and the isometry therefore give
\[
 \mathbb E\sup_{0\leq t\leq T}
    \|\Psi(t)-\Psi^\varepsilon(t)\|_{H^3}^2
 \leq4\int_0^T\int_{\{|z|\leq\varepsilon\}}
                 \|Gz\|_{H^3}^2\nu(\mathrm dz)\,\mathrm ds
 =4Tq_3(\varepsilon).
\]
For two cutoffs, the same integral is restricted to
\(\{\min(\varepsilon,\eta)<|z|\leq\max(\varepsilon,\eta)\}\),
up to its sign. Its quadratic mean is bounded by
\(Tq_3(\max\{\varepsilon,\eta\})\), which proves
\eqref{prob:convolutionannulus}.
This completes the proof.
\end{proof}

The next estimate records the spatial regularity used in the
deterministic forcing argument. Recall that
\[
 \|f\|_{Z(I)}
 =\left(\sum_{N\ge1}N^2
   \|P_Nf\|_{L^{82}(I\times\mathbb T^3)}^2\right)^{1/2},
\]
where the sum is over dyadic \(N\). This definition is used on arbitrary
bounded intervals.

\begin{lemma}\label{prob:Zembed}
For any bounded interval \(I\) and strongly measurable
\(f\in L^{82}(I;H^3)\),
\begin{equation}\label{prob:Zembedding}
 \|f\|_{Z(I)}\lesssim\|f\|_{L^{82}(I;H^3)}.
\end{equation}
Consequently,
\begin{equation}\label{prob:Zsup}
 \|f\|_{Z(I)}
 \lesssim |I|^{1/82}\sup_{t\in I}\|f(t)\|_{H^3}
\end{equation}
whenever the right-hand side is finite. The same estimates hold for the
difference of two such functions.
\end{lemma}

\begin{proof}
Bernstein's inequality and the frequency localization of \(P_N\)
give
\[
 N\|P_Nf(t)\|_{L^{82}}
 \lesssim N^{1+3(1/2-1/82)}\|P_Nf(t)\|_{L^2}
 \lesssim N^{-1/2-3/82}\|f(t)\|_{H^3}.
\]
Taking the \(L^{82}(I)\) norm, squaring, and summing over dyadic
\(N\) yields
\[
 \|f\|_{Z(I)}^2
 \lesssim\sum_{N\geq1}N^{-1-6/82}
                 \|f\|_{L^{82}(I;H^3)}^2
 \lesssim\|f\|_{L^{82}(I;H^3)}^2.
\]
Moreover,
\[
 \|f\|_{L^{82}(I;H^3)}
 =\left(\int_I\|f(t)\|_{H^3}^{82}\,\mathrm dt\right)^{1/82}
 \leq |I|^{1/82}\sup_{t\in I}\|f(t)\|_{H^3}.
\]
The estimates apply to \(f-g\) whenever both functions satisfy
the stated assumptions.
This completes the proof.
\end{proof}

It follows that, outside one null set, every bounded time interval
satisfies
\[
 \|\Psi\|_{L^\infty(I;H^3)}+\|\Psi\|_{Z(I)}<\infty.
\]
For each truncation, these bounds follow by applying
\eqref{prob:Zsup} to its c\`adl\`ag \(H^3\) paths.

\subsection{Global solutions with finitely many retained jumps}

\begin{proposition}
\label{prob:cutoff}
For each \(0<\varepsilon<1\), there is a global adapted
\(H^1\)-valued c\`adl\`ag process \(u^\varepsilon\) satisfying
\begin{equation}\label{prob:cutoffmild}
 u^\varepsilon(t)
 =S(t)u_0-i\int_0^tS(t-s)\mathcal N(u^\varepsilon(s))\,\mathrm ds
       +\Psi^\varepsilon(t).
\end{equation}
On every bounded time interval,
\begin{equation}\label{prob:cutoffremainder}
 v^\varepsilon:=u^\varepsilon-\Psi^\varepsilon
 \in X^1([0,T])\subset C([0,T];H^1)
 \quad\hbox{almost surely}.
\end{equation}
At each retained jump with mark \(z\),
\begin{equation}\label{prob:cutoffjump}
 u^\varepsilon(t)=u^\varepsilon(t-)+Gz.
\end{equation}
\end{proposition}

\begin{proof}
The second-moment assumption yields
\[
 \nu(Z_\varepsilon)\le\varepsilon^{-2}m_2<\infty.
\]
Thus the retained jump times have no finite accumulation point.
There are almost surely no simultaneous retained jumps, since the time
part of the intensity is Lebesgue measure.
Symmetry of \(\nu\) gives
\begin{equation}\label{prob:zerodrift}
 \int_{Z_\varepsilon}z\,\nu(\mathrm dz)=0,\qquad
 \int_{Z_\varepsilon}Gz\,\nu(\mathrm dz)=0.
\end{equation}
Both integrals exist absolutely because \(Z_\varepsilon\) has finite
measure and bounded marks.

Let \(\mathcal U(t,w)\) denote the global deterministic defocusing
quintic flow at elapsed time \(t\geq0\) with initial value \(w\in H^1\),
provided by \cite{IP2012}. On a fixed interval \([0,T]\), write the
retained times as \(0<t_1<\cdots<t_K\leq T\), with marks
\(z_1,\ldots,z_K\), and put \(t_0=0\), \(t_{K+1}=T\), and \(w_0=u_0\).
Define recursively
\[
 w_j=\mathcal U(t_j-t_{j-1},w_{j-1})+Gz_j,
 \qquad 1\leq j\leq K.
\]
For \(t_j\leq t<t_{j+1}\), set
\(u^\varepsilon(t)=\mathcal U(t-t_j,w_j)\); the last segment
extends from \(t_K\) to \(T\), including \(T\).
When \(K=0\), use \(\mathcal U(t,u_0)\) on the whole interval.
The continuity of the deterministic flow in \(H^1\) gives
\[
 u^\varepsilon(t_j-)=\mathcal U(t_j-t_{j-1},w_{j-1}),
 \qquad u^\varepsilon(t_j)=u^\varepsilon(t_j-)+Gz_j.
\]
The definitions agree as \(T\) increases and give a global
c\`adl\`ag process. The map \((t,w)\mapsto\mathcal U(t,w)\) is
continuous, and each retained time and mark is measurable with
respect to the Poisson random measure up to that time.
The recursion therefore makes \(u^\varepsilon(t)\)
\(\mathcal F_t\)-measurable for every \(t\geq0\).

By \eqref{prob:zerodrift}, the convolution on a bounded interval has the
pathwise representation
\begin{equation}\label{prob:finiteconvolution}
 \Psi^\varepsilon(t)=
 \sum_{\substack{0<s\leq t\\ \text{retained jump at }s}}
            S(t-s)Gz_s.
\end{equation}
Here $z_s$ is the mark of the retained atom $(s,z_s)$ of $N$.
On the segment starting at \(t_j\), variation of constants gives
\[
 u^\varepsilon(t)=S(t-t_j)w_j
 -i\int_{t_j}^tS(t-s)\mathcal N(u^\varepsilon(s))\,\mathrm ds.
\]
Substituting the recursion for \(w_j\) and using
\(S(t-t_j)S(t_j-s)=S(t-s)\), induction over \(j\) yields
\[
 u^\varepsilon(t)=S(t)u_0
 -i\int_0^tS(t-s)\mathcal N(u^\varepsilon(s))\,\mathrm ds
 +\sum_{t_j\leq t}S(t-t_j)Gz_j.
\]
Together with \eqref{prob:finiteconvolution}, this is
\eqref{prob:cutoffmild}. At every retained time \(t_j\),
\[
 v^\varepsilon(t_j)-v^\varepsilon(t_j-)=Gz_j-Gz_j=0.
\]
Thus \(v^\varepsilon\) is continuous on \([0,T]\).

On each segment \(I_j=[t_j,t_{j+1}]\),
\(0\leq j\leq K\), use the pre-jump values
at the right endpoint. The deterministic solution belongs to
\(X^1(I_j)\), while the convolution on this segment is
\(S(t-t_j)\Psi^\varepsilon(t_j)\). Hence the free evolution estimate
gives
\[
 \|v^\varepsilon\|_{X^1(I_j)}
 \leq\|\mathcal U(\,\cdot-t_j,w_j)\|_{X^1(I_j)}
       +\|\Psi^\varepsilon(t_j)\|_{H^1}<\infty.
\]
Here and for the last segment, a segment of length zero is omitted.
The continuous remainders have matching endpoint values because
their jumps vanish. Lemma~\ref{det:concatenation} therefore gives
\[
 \|v^\varepsilon\|_{X^1([0,T])}
 \leq C_{\mathrm{con}}\sum_j
                  \|v^\varepsilon\|_{X^1(I_j)}<\infty,
\]
which proves \eqref{prob:cutoffremainder}.
This completes the proof.
\end{proof}

\subsection{Uniform energy moments from finite-jump identities}

Recall the functional \(F\) in \eqref{eq:energyfunctional}. Thus
\[
 F(w)=1+\|w\|_{L^2}^2+\frac12\|\nabla w\|_{L^2}^2
                      +\frac16\|w\|_{L^6}^6.
\]
We will use repeatedly that
\begin{equation}\label{prob:Fcoercivity}
 1+\frac12\|w\|_{H^1}^2\le F(w)
 \lesssim\bigl(1+\|w\|_{H^1}^6\bigr).
\end{equation}
The lower bound follows by retaining the quadratic terms of \(F\);
the upper bound uses \(H^1(\mathbb T^3)\hookrightarrow L^6(\mathbb T^3)\).
The mass and energy conservation laws make \(F(u^\varepsilon)\)
constant between consecutive retained jumps.

\begin{lemma}\label{prob:taylor}
For \(w\in H^1\) and \(z\in Z_0\), put
\begin{align}
 h(w,z)&=F(w+Gz)-F(w),\label{prob:hdef}\\
 R(w,z)&=F(w+Gz)-F(w)-DF(w)[Gz].\label{prob:Rdef}
\end{align}
There is a finite constant \(C_G\), depending only on
\(\|G\|_{\mathcal L(\mathbb R^m,H^1)}\) and
\(\|G\|_{\mathcal L(\mathbb R^m,L^\infty)}\), such that
\begin{align}
 |DF(w)[Gz]|&\le C_G|z|F(w),\label{prob:firstincrement}\\
 0\le R(w,z)&\le C_G|z|^2F(w),\label{prob:secondincrement}\\
 |h(w,z)|&\le C_G|z|F(w).\label{prob:fullincrement}
\end{align}
After increasing \(C_G\) if necessary, there is \(A_G\ge1\) such that
\begin{equation}\label{prob:overshoot}
 F(w+Gz)\le A_GF(w),\qquad w\in H^1,\ z\in Z_0.
\end{equation}
\end{lemma}

\begin{proof}
All derivatives are real Fr\'echet derivatives on the complex Sobolev
space. Since \(H^1\hookrightarrow L^6\), the functional \(F\) is a
continuous polynomial of degree six in the real and imaginary parts
of \(w\). In particular, the following derivatives and Taylor formula
are valid directly on \(H^1\).
Writing \(g=Gz\), we have
\begin{equation}\label{prob:DF}
 \begin{split}
 DF(w)[g]={}&
 2\operatorname{Re}\int_{\mathbb T^3}\overline w g
 +\operatorname{Re}\int_{\mathbb T^3}
                  \nabla\overline w\cdot\nabla g\\
 &+\operatorname{Re}\int_{\mathbb T^3}|w|^4\overline w g.
 \end{split}
\end{equation}
Cauchy--Schwarz and the boundedness of \(G\) give
\[
 \left|2\operatorname{Re}\int\overline w g
       +\operatorname{Re}\int\nabla\overline w\cdot\nabla g\right|
 \lesssim_G |z|\|w\|_{H^1}
 \lesssim_G |z|F(w)^{1/2}.
\]
For the last term, H\"older's inequality gives
\[
 \left|\int |w|^4\overline w g\right|
 \le\|w\|_{L^6}^5\|g\|_{L^6}
 \lesssim_G |z|F(w)^{5/6}.
\]
As \(F(w)\ge1\), these estimates prove
\eqref{prob:firstincrement}.

The second derivative in the repeated direction \(g\) is
\begin{equation}\label{prob:D2F}
 \begin{split}
 D^2F(w)[g,g]={}&2\|g\|_{L^2}^2+\|\nabla g\|_{L^2}^2\\
 &+\int_{\mathbb T^3}|w|^4|g|^2
   +4\int_{\mathbb T^3}|w|^2
             \bigl(\operatorname{Re}(\overline wg)\bigr)^2.
 \end{split}
\end{equation}
Every term is nonnegative, and the last two terms together are at most
\(5\|g\|_{L^\infty}^2\|w\|_{L^4}^4\).
For \(0\leq s\leq1\), the normalized volume of the torus and
\(|z|\leq1\) give
\[
 \|w+sg\|_{L^4}^4
 \leq8\bigl(\|w\|_{L^4}^4+\|g\|_{L^4}^4\bigr)
 \lesssim_G F(w)^{2/3}+|z|^4
 \lesssim_G F(w).
\]
Consequently,
\[
 \begin{split}
 D^2F(w+sg)[g,g]
 &\leq2\|g\|_{H^1}^2
       +5\|g\|_{L^\infty}^2\|w+sg\|_{L^4}^4\\
 &\lesssim_G |z|^2F(w).
 \end{split}
\]
The integral remainder formula reads
\[
 R(w,z)=\int_0^1(1-s)D^2F(w+sGz)[Gz,Gz]\,\mathrm ds.
\]
It proves the nonnegativity of \(R\) and the upper bound
\eqref{prob:secondincrement}. Combining the first two bounds and using
\(|z|^2\le|z|\) gives \eqref{prob:fullincrement}. Finally,
\[
 F(w+Gz)=F(w)+h(w,z)\le(1+C_G)F(w),
\]
which proves \eqref{prob:overshoot}.
This completes the proof.
\end{proof}

\begin{proposition}\label{prob:energy}
For every \(T<\infty\), there is
\[
 C_T=C\bigl(T,m_2,
     \|G\|_{\mathcal L(\mathbb R^m,H^1)},
     \|G\|_{\mathcal L(\mathbb R^m,L^\infty)}\bigr)<\infty
\]
such that
\begin{equation}\label{prob:maxF}
 \sup_{0<\varepsilon<1}
 \mathbb E\sup_{0\leq t\leq T}F(u^\varepsilon(t))^2
 \le C_TF(u_0)^2.
\end{equation}
The constants can be chosen uniformly if \(G\) is replaced by
\(\lambda G\), \(0\le\lambda\le1\).
\end{proposition}

\begin{proof}
Fix a cutoff and abbreviate
\[
 F_t^\varepsilon=F(u^\varepsilon(t)),\qquad F_0=F(u_0).
\]
The conservation laws on the deterministic segments and the finite
number of retained jumps give the exact pathwise identity
\begin{equation}\label{prob:finitebalance}
 F_t^\varepsilon
 =F_0+\int_0^t\int_{Z_\varepsilon}
          h(u^\varepsilon(s-),z)\,N(\mathrm ds,\mathrm dz).
\end{equation}
This identity holds for \(H^1\) data: the deterministic flow conserves
the mass and energy at that regularity, and each term of the finite
jump sum is a well-defined difference of continuous \(H^1\)
functionals.

Define
\begin{equation}\label{prob:adrift}
 a_\varepsilon(w)=
       \int_{Z_\varepsilon}R(w,z)\,\nu(\mathrm dz).
\end{equation}
The integrable symmetry cancellation \eqref{prob:zerodrift} and the
real linearity of \(DF(w)\) give
\[
 \int_{Z_\varepsilon}h(w,z)\,\nu(\mathrm dz)
 =DF(w)\left[\int_{Z_\varepsilon}Gz\,\nu(\mathrm dz)\right]
    +\int_{Z_\varepsilon}R(w,z)\,\nu(\mathrm dz)
 =a_\varepsilon(w).
\]
Thus \eqref{prob:finitebalance} can be written, initially after stopping,
as
\begin{equation}\label{prob:compensatedbalance}
 F_t^\varepsilon
 =F_0+\int_0^t a_\varepsilon(u^\varepsilon(s-))\,\mathrm ds
       +M_t^\varepsilon,
\end{equation}
where
\begin{equation}\label{prob:energymartingale}
 M_t^\varepsilon=
  \int_0^t\int_{Z_\varepsilon}
        h(u^\varepsilon(s-),z)\,\widetilde N(\mathrm ds,\mathrm dz).
\end{equation}
Lemma~\ref{prob:taylor} implies that a constant \(b\ge0\), independent
of \(\varepsilon\), can be chosen so that
\begin{equation}\label{prob:driftbracket}
 0\le a_\varepsilon(w)\le bF(w),\qquad
 \int_{Z_\varepsilon}|h(w,z)|^2\nu(\mathrm dz)\le bF(w)^2.
\end{equation}
For example, \(b=C_Gm_2\) is sufficient after increasing \(C_G\).
We first obtain a fixed-time second moment by localization.
For an integer \(n>F_0\), let
\[
 \tau_n=\inf\{t\ge0:F_t^\varepsilon\ge n\}.
\]
Since \(F^\varepsilon\) is constant between retained jumps,
\(\tau_n\) is the first retained jump time at which its value is
at least \(n\), and is therefore a stopping time.
Before \(\tau_n\), \(F^\varepsilon<n\).
At the stopping jump, \eqref{prob:overshoot} gives
\begin{equation}\label{prob:stoppedbound}
 \sup_{t\ge0}F_{t\wedge\tau_n}^\varepsilon\le A_Gn.
\end{equation}
This bound makes the compensated integral stopped at \(\tau_n\)
square integrable.
The stopping indicator \(\mathbf1_{\{0<s\leq\tau_n\}}\) is predictable;
it includes the jump at the stopping time.

The scalar jump product formula for \((F^\varepsilon)^2\) gives
\begin{equation}\label{prob:squarestopped}
 \begin{split}
 \mathbb E(F_{t\wedge\tau_n}^\varepsilon)^2
 =F_0^2+\mathbb E\int_0^t\mathbf1_{\{0<s\leq\tau_n\}}
 \bigg[&
 2F_{s-}^\varepsilon a_\varepsilon(u^\varepsilon(s-))\\
 &+\int_{Z_\varepsilon}
     |h(u^\varepsilon(s-),z)|^2\nu(\mathrm dz)
 \bigg]\mathrm ds.
 \end{split}
\end{equation}
Indeed, the change in the square at a jump is
\[
 (F_{s-}^\varepsilon+h)^2-(F_{s-}^\varepsilon)^2
 =2F_{s-}^\varepsilon h+h^2.
\]
After compensating the jump sum, its drift is the integrand of
\eqref{prob:squarestopped}, and its martingale integrand is
\[
 2F_{s-}^\varepsilon h(u^\varepsilon(s-),z)
       +h(u^\varepsilon(s-),z)^2.
\]
On \(\{0<s\leq\tau_n\}\), its absolute value is at most
\(C_Gn^2|z|\), using \(|z|\le1\).
Its second moment is therefore integrable, which justifies the
zero expectation used in \eqref{prob:squarestopped}.

By \eqref{prob:driftbracket}, the bracketed drift is at most
\(3b(F_{s-}^\varepsilon)^2\). With
\(y_n(t)=\mathbb E(F_{t\wedge\tau_n}^\varepsilon)^2\), this implies
\[
 y_n(t)\le F_0^2+3b\int_0^t y_n(s)\,\mathrm ds.
\]
To justify the integral comparison, for Lebesgue almost every \(s\)
on each path,
\[
 \mathbf1_{\{0<s\leq\tau_n\}}(F_{s-}^\varepsilon)^2
 =\mathbf1_{\{0<s\leq\tau_n\}}(F_s^\varepsilon)^2
 \leq(F_{s\wedge\tau_n}^\varepsilon)^2.
\]
Tonelli's theorem then bounds its expected integral by
\(\int_0^t y_n(s)\,\mathrm ds\).
Gronwall's inequality yields
\begin{equation}\label{prob:stoppedsecond}
 \mathbb E(F_{t\wedge\tau_n}^\varepsilon)^2
 \le F_0^2e^{3bt}.
\end{equation}

For a fixed cutoff, the global solution has finitely many jumps on
\([0,T]\). Writing \(K_T^\varepsilon=N((0,T]\times Z_\varepsilon)\),
iteration of \eqref{prob:overshoot} gives
\[
 \sup_{0\leq t\leq T}F_t^\varepsilon
 \leq A_G^{K_T^\varepsilon}F_0<\infty
 \quad\hbox{almost surely}.
\]
Thus \(\tau_n>T\) for all sufficiently large \(n\), almost surely.
Taking integer \(T\) yields \(\tau_n\to\infty\) almost surely.
Fatou's lemma applied to \eqref{prob:stoppedsecond} gives
\begin{equation}\label{prob:fixedsecond}
 \mathbb E(F_t^\varepsilon)^2\le F_0^2e^{3bt},
 \qquad t\ge0.
\end{equation}
The bound is uniform in the cutoff.

We can now remove the localization from the martingale estimate.
The predictable bracket integrability follows from
\eqref{prob:driftbracket} and \eqref{prob:fixedsecond}:
\[
 \mathbb E\int_0^T\int_{Z_\varepsilon}
       |h(u^\varepsilon(s-),z)|^2\nu(\mathrm dz)\,\mathrm ds
 \le bF_0^2\int_0^T e^{3bs}\,\mathrm ds.
\]
Thus \(M^\varepsilon\) in \eqref{prob:energymartingale} is a
square-integrable martingale on \([0,T]\), and Doob's inequality gives
\begin{equation}\label{prob:energydoob}
 \mathbb E\sup_{0\leq t\leq T}|M_t^\varepsilon|^2
 \le4bF_0^2\int_0^T e^{3bs}\,\mathrm ds.
\end{equation}
Finally, apply \((x+y+z)^2\le3(x^2+y^2+z^2)\) to
\eqref{prob:compensatedbalance} and use Cauchy--Schwarz in time:
\[
 \begin{split}
 \mathbb E\sup_{0\leq t\leq T}(F_t^\varepsilon)^2
 &\le3F_0^2
   +3Tb^2\int_0^T\mathbb E(F_s^\varepsilon)^2\,\mathrm ds
   +3\mathbb E\sup_{0\leq t\leq T}|M_t^\varepsilon|^2\\
 &\le
 \left[3+(3Tb^2+12b)\int_0^T e^{3bs}\,\mathrm ds\right]F_0^2.
 \end{split}
\]
This proves \eqref{prob:maxF}. Each Taylor estimate can use the same
upper constant for \(\lambda G\), \(0\le\lambda\le1\), proving the last
assertion.
This completes the proof.
\end{proof}

\subsection{Measurability of the critical remainders}

The uniform distance between two c\`adl\`ag \(H^1\)-valued paths is
measurable: its supremum can be taken over a fixed countable dense set
of times, together with the right endpoint. For the \(X^1\) distances,
the following lemma places the approximating remainders in a common
separable subspace.

\begin{lemma}
\label{prob:measurable}
Fix \(T<\infty\) and \(u_0\in H^1\). There is a closed separable
linear subspace \(\mathcal X_T\) of \(X^1([0,T])\) such that every
\(v^\varepsilon\) has a Borel measurable \(\mathcal X_T\)-valued
realization. In particular, the random variables
\[
 \|v^\varepsilon\|_{X^1([0,T])},\qquad
 \|v^\varepsilon-v^\eta\|_{X^1([0,T])}
\]
are measurable.
\end{lemma}

\begin{proof}
For a nonnegative integer \(\ell\), consider the finite-dimensional
parameter set
\[
 \mathcal P_\ell=
 \left\{(t_1,\ldots,t_\ell,z_1,\ldots,z_\ell):
       0<t_1<\cdots<t_\ell<T,\quad z_j\in Z_0\right\}.
\]
Equip $\mathcal P_\ell$ with the topology inherited from
$\mathbb R^{\ell(1+m)}$ and its Borel sigma field.
For \(\ell=0\), this is a singleton. Given \(p\in\mathcal P_\ell\), put
\begin{equation}\label{prob:parameterforcing}
 \Psi_p(t)=\sum_{j=1}^{\ell}
         \mathbf1_{[t_j,T]}(t)S(t-t_j)Gz_j.
\end{equation}
Let \(u_p\) be the deterministic finite-jump construction with these
times and marks, and let \(V_\ell(p)=u_p-\Psi_p\).
The proof of Proposition~\ref{prob:cutoff} gives
\(V_\ell(p)\in X^1([0,T])\).

First, let \(p_n\to p\) in \(\mathcal P_\ell\), and denote
their respective coordinates by \(t_{j,n},z_{j,n}\) and \(t_j,z_j\).
On the interval where both \(j\)-th summands in
\eqref{prob:parameterforcing} are present, unitarity gives
\[
 \begin{split}
 &\|S(t-t_{j,n})Gz_{j,n}-S(t-t_j)Gz_j\|_{H^3}\\
 &\qquad\leq\|G(z_{j,n}-z_j)\|_{H^3}
       +\|(S(t_j-t_{j,n})-\mathrm{Id})Gz_j\|_{H^3}.
 \end{split}
\]
Where exactly one summand is present, its norm is at most
\(\|G\|_{\mathcal L(\mathbb R^m,H^3)}\), and that interval has length
\(|t_{j,n}-t_j|\). The triangle inequality in \(L^{82}([0,T];H^3)\)
therefore yields
\[
 \begin{split}
 \|\Psi_{p_n}-\Psi_p\|_{L^{82}([0,T];H^3)}
 \leq{}&T^{1/82}\sum_{j=1}^{\ell}
       \bigl(\|G(z_{j,n}-z_j)\|_{H^3}\\
 &\hspace{32mm}
       +\|(S(t_j-t_{j,n})-\mathrm{Id})Gz_j\|_{H^3}\bigr)\\
 &+\|G\|_{\mathcal L(\mathbb R^m,H^3)}
               \sum_{j=1}^{\ell}|t_{j,n}-t_j|^{1/82}
 \longrightarrow0.
 \end{split}
\]
The convergence follows from the boundedness of \(G\) and the
strong continuity of \(S\). Lemma~\ref{prob:Zembed} gives
\(\|\Psi_{p_n}-\Psi_p\|_{Z([0,T])}\to0\).

Second, the bounds required in the deterministic stability theorem
are uniform over \(\mathcal P_\ell\). Conservation between the prescribed
jumps and \eqref{prob:overshoot} imply
\[
 \sup_{0\leq t\leq T}F(u_p(t))\le A_G^\ell F(u_0).
\]
Also,
\[
 \sup_{0\leq t\leq T}\|\Psi_p(t)\|_{H^3}
 \leq\ell\|G\|_{\mathcal L(\mathbb R^m,H^3)}.
\]
By \eqref{prob:Fcoercivity} and \(V_\ell(p)=u_p-\Psi_p\),
\[
 \sup_{0\leq t\leq T}\|V_\ell(p)(t)\|_{H^1}
 \leq\bigl(2A_G^\ell F(u_0)\bigr)^{1/2}
             +\ell\|G\|_{\mathcal L(\mathbb R^m,H^3)}.
\]
Both bounds are independent of \(p\in\mathcal P_\ell\).
Apply the \(Z\)-distance version of
Theorem~\ref{det:stability} to the two continuous remainders.
The uniform bounds just obtained give
\[
 \|V_\ell(p_n)-V_\ell(p)\|_{X^1([0,T])}
 \lesssim_{T,\ell,u_0,G}
       \|\Psi_{p_n}-\Psi_p\|_{Z([0,T])}\longrightarrow0.
\]
Thus \(V_\ell\) is continuous.

Each \(\mathcal P_\ell\) is a separable metric space, so its continuous
image in \(X^1([0,T])\) is separable.
Let \(\mathcal X_T\) be the closed linear span of
\(\bigcup_{\ell\ge0}V_\ell(\mathcal P_\ell)\).
It is a closed separable subspace of the Banach space \(X^1([0,T])\),
equipped with the inherited norm and its Borel sigma field.
For a fixed cutoff, put
\[
 K_t^\varepsilon=N((0,t]\times Z_\varepsilon),\qquad
 T_j^\varepsilon=\inf\{t\geq0:K_t^\varepsilon\geq j\},\quad j\geq1.
\]
The arrival times \(T_j^\varepsilon\) are stopping times, and the
associated marks are measurable functions of the Poisson random
measure. Since the intensity in time is diffuse, almost surely
the retained arrivals are distinct and neither endpoint \(0,T\)
is an arrival time. On the event \(\{K_T^\varepsilon=\ell\}\),
the ordered times and marks thus form a measurable random element
of \(\mathcal P_\ell\), and the remainder equals \(V_\ell\)
applied to this element.
For an open set \(O\subset\mathcal X_T\), the event
\(\{v^\varepsilon\in O\}\), up to the exceptional null set, is
the countable union
\[
 \bigcup_{\ell\geq0}
 \left\{K_T^\varepsilon=\ell,\
   (T_1^\varepsilon,\ldots,T_\ell^\varepsilon,
        z_1^\varepsilon,\ldots,z_\ell^\varepsilon)
                   \in V_\ell^{-1}(O)\right\},
\]
where \(z_j^\varepsilon\) denotes the mark at \(T_j^\varepsilon\).
Each event is measurable by the continuity of \(V_\ell\).
Assigning \(v^\varepsilon=0\) on the exceptional null set gives
the asserted Borel realization. The norm is continuous on
\(\mathcal X_T\), as is \((x,y)\mapsto\|x-y\|_{X^1([0,T])}\),
which proves the measurability of both displayed random variables.
This completes the proof.
\end{proof}

\subsection{Cauchy estimates on bounded events}

We next combine the energy bound with the deterministic bounded-set
estimates on events where the energy and convolution norms are bounded.

\begin{proposition}
\label{prob:cauchy}
For every \(T<\infty\), the family
\(\{\|v^\varepsilon\|_{X^1([0,T])}:0<\varepsilon<1\}\)
is bounded in probability. Moreover, if
\begin{equation}\label{prob:distance}
 d_T(\varepsilon,\eta)=
 \sup_{0\leq t\leq T}\|u^\varepsilon(t)-u^\eta(t)\|_{H^1}
 +\|v^\varepsilon-v^\eta\|_{X^1([0,T])},
\end{equation}
then
\begin{equation}\label{prob:cauchystatement}
 d_T(\varepsilon,\eta)\longrightarrow0
 \quad\hbox{in probability as }\max\{\varepsilon,\eta\}\downarrow0.
\end{equation}
More precisely, there are \(C_{T,u_0,G,\nu}<\infty\) and, for each
\(R\ge1\), \(K_T(R)<\infty\) such that
\begin{equation}\label{prob:cauchyquantitative}
 \mathbb P\bigl(d_T(\varepsilon,\eta)>\rho\bigr)
 \le \frac{C_{T,u_0,G,\nu}}{R^2}
   +\frac{4TK_T(R)^2}{\rho^2}\,
        q_3(\max\{\varepsilon,\eta\}),\qquad \rho>0.
\end{equation}
\end{proposition}

\begin{proof}
Let \(\mathcal A_R^{\varepsilon,\eta}\) be the event on which all four
quantities
\[
 \sup_{0\leq t\leq T}F(u^\varepsilon(t)),\quad
 \sup_{0\leq t\leq T}F(u^\eta(t)),\quad
 \sup_{0\leq t\leq T}\|\Psi^\varepsilon(t)\|_{H^3},\quad
 \sup_{0\leq t\leq T}\|\Psi^\eta(t)\|_{H^3}
\]
are at most \(R\). These are measurable by adapted c\`adl\`ag
path regularity. Proposition~\ref{prob:energy},
Lemma~\ref{prob:convolution}, and Markov's inequality give
\begin{equation}\label{prob:goodevent}
 \mathbb P\bigl((\mathcal A_R^{\varepsilon,\eta})^c\bigr)
 \le \frac{2C_TF(u_0)^2+8Tq_3}{R^2}.
\end{equation}
On \(\mathcal A_R^{\varepsilon,\eta}\),
\[
 \sup_{0\leq t\leq T}\|v^\varepsilon(t)\|_{H^1}
   +\sup_{0\leq t\leq T}\|v^\eta(t)\|_{H^1}
 \le2\sqrt{2R}+2R.
\]
Put \(R_*=\sqrt{2R}+R\), and take
\(\delta=\delta(R_*,R)\) and \(A=A(R_*,R)\) from
Proposition~\ref{det:restart}. Partition \([0,T]\) into at most
\(n_T=1+\lceil T/\delta\rceil\) intervals of length at most
\(\delta\). At each left endpoint, either remainder has \(H^1\)
norm at most \(R_*\). Local uniqueness identifies its restriction
with the solution supplied by the restart proposition, so its
\(X^1\) norm on that interval is at most \(A\).
Lemma~\ref{det:concatenation} then bounds its whole-interval norm
by
\[
 A_T(R):=C_{\mathrm{con}}n_TA.
\]
For a single cutoff these conclusions require only the event
\[
 \left\{\sup_{0\leq t\leq T}F(u^\varepsilon(t))\leq R,\quad
        \sup_{0\leq t\leq T}\|\Psi^\varepsilon(t)\|_{H^3}\leq R\right\}.
\]
Its complement has probability at most
\((C_TF(u_0)^2+4Tq_3)/R^2\), and hence
\[
 \sup_{0<\varepsilon<1}
 \mathbb P\bigl(\|v^\varepsilon\|_{X^1([0,T])}>A_T(R)\bigr)
 \leq \frac{C_TF(u_0)^2+4Tq_3}{R^2}.
\]
Given any positive probability threshold, choosing \(R\) large
enough proves boundedness in probability.

The initial values of the two remainders agree. The
\(L^\infty H^3\)-distance version of
Theorem~\ref{det:stability} gives, on
\(\mathcal A_R^{\varepsilon,\eta}\),
\[
 \|v^\varepsilon-v^\eta\|_{X^1([0,T])}
 \le C_T(R)\sup_{0\leq t\leq T}
          \|\Psi^\varepsilon(t)-\Psi^\eta(t)\|_{H^3}.
\]
The identity \(u^\varepsilon=v^\varepsilon+\Psi^\varepsilon\)
and \eqref{det:linear} give
\[
 \begin{split}
 \sup_{0\leq t\leq T}\|u^\varepsilon(t)-u^\eta(t)\|_{H^1}
 &\leq\|v^\varepsilon-v^\eta\|_{X^1([0,T])}\\
 &\quad+\sup_{0\leq t\leq T}
                 \|\Psi^\varepsilon(t)-\Psi^\eta(t)\|_{H^3}.
 \end{split}
\]
Adding the remainder distance, and setting \(K_T(R)=2C_T(R)+1\),
we obtain
\begin{equation}\label{prob:pathwisecomparison}
 d_T(\varepsilon,\eta)
 \le K_T(R)\sup_{0\leq t\leq T}
          \|\Psi^\varepsilon(t)-\Psi^\eta(t)\|_{H^3}
 \quad\hbox{on }\mathcal A_R^{\varepsilon,\eta}.
\end{equation}
By Lemma~\ref{prob:measurable}, the left-hand side is a random variable.
The event inclusion
\[
 \{d_T(\varepsilon,\eta)>\rho\}
 \subset(\mathcal A_R^{\varepsilon,\eta})^c
 \cup\left\{
   K_T(R)\sup_{0\leq t\leq T}
       \|\Psi^\varepsilon(t)-\Psi^\eta(t)\|_{H^3}>\rho
      \right\}
\]
and Chebyshev's inequality yield
\[
 \begin{split}
 \mathbb P(d_T(\varepsilon,\eta)>\rho)
 &\leq\mathbb P((\mathcal A_R^{\varepsilon,\eta})^c)\\
 &\quad+\frac{K_T(R)^2}{\rho^2}
      \mathbb E\sup_{0\leq t\leq T}
       \|\Psi^\varepsilon(t)-\Psi^\eta(t)\|_{H^3}^2\\
 &\leq\frac{2C_TF(u_0)^2+8Tq_3}{R^2}
     +\frac{4TK_T(R)^2}{\rho^2}
                     q_3(\max\{\varepsilon,\eta\}).
 \end{split}
\]
This proves \eqref{prob:cauchyquantitative}. For every fixed
\(\rho>0\) and \(R\geq1\), letting
\(\max\{\varepsilon,\eta\}\downarrow0\) gives
\[
 \limsup_{\max\{\varepsilon,\eta\}\downarrow0}
       \mathbb P(d_T(\varepsilon,\eta)>\rho)
 \leq\frac{2C_TF(u_0)^2+8Tq_3}{R^2}.
\]
Letting \(R\to\infty\) proves \eqref{prob:cauchystatement}.
This completes the proof.
\end{proof}

\subsection{Identification of the global limit}

\begin{proposition}
\label{prob:limit}
There are adapted processes \(u\) and \(v\), defined for all \(t\ge0\),
such that, almost surely, for every \(T<\infty\),
\begin{equation}\label{prob:limitclass}
 u\in D([0,T];H^1),\qquad v=u-\Psi\in X^1([0,T]).
\end{equation}
They satisfy the mild equation \eqref{eq:mild}, and
\begin{equation}\label{prob:allcutoffs}
 \sup_{0\leq t\leq T}\|u^\varepsilon(t)-u(t)\|_{H^1}
    +\|v^\varepsilon-v\|_{X^1([0,T])}
 \longrightarrow0
 \quad\hbox{in probability as }\varepsilon\downarrow0.
\end{equation}
Moreover,
\begin{equation}\label{prob:limitenergy}
 \mathbb E\sup_{0\leq t\leq T}F(u(t))^2\le C_TF(u_0)^2,
\end{equation}
with the constant of Proposition~\ref{prob:energy}.
\end{proposition}

\begin{proof}
By Proposition~\ref{prob:cauchy} and
\eqref{prob:convolutiontail}, one can choose a deterministic decreasing
sequence \(\varepsilon_n\downarrow0\) such that
\begin{align}
 \mathbb P\bigl(d_n(\varepsilon_n,\varepsilon_{n+1})>2^{-n}\bigr)
 &\le2^{-n},\label{prob:subsequencechoice1}\\
 \mathbb P\left(
   \sup_{0\leq t\leq n}\|\Psi^{\varepsilon_n}(t)-\Psi(t)\|_{H^3}>2^{-n}
          \right)
 &\le2^{-n}.\label{prob:subsequencechoice2}
\end{align}
Here \(d_n\) means \eqref{prob:distance} with the time horizon \(T=n\).
To specify the choice, for each \(n\geq1\) choose \(\delta_n>0\)
such that
\[
 \mathbb P\bigl(d_n(\varepsilon,\eta)>2^{-n}\bigr)\leq2^{-n}
 \qquad(0<\varepsilon,\eta<\delta_n).
\]
Choose also \(0<\gamma_n<1\) with
\[
 4n\,2^{2n}q_3(\gamma_n)\leq2^{-n}.
\]
Starting from \(\varepsilon_0=1\), select recursively
\[
 0<\varepsilon_n<
 \min\{\delta_n,\gamma_n,2^{-n},\varepsilon_{n-1}/2\}.
\]
Both \(\varepsilon_n\) and \(\varepsilon_{n+1}\) are below
\(\delta_n\), which gives \eqref{prob:subsequencechoice1}.
The convolution estimate \eqref{prob:convolutiontail} and Chebyshev's
inequality give \eqref{prob:subsequencechoice2}.

By Borel--Cantelli, outside one null set the two events in
\eqref{prob:subsequencechoice1}--\eqref{prob:subsequencechoice2}
occur for only finitely many \(n\).
Fix an integer \(k\). For \(n\geq k\), the definitions of the
uniform distance and the \(X^1\) restriction norm give
\[
 \begin{split}
 &\sup_{0\leq t\leq k}\|u^{\varepsilon_{n+1}}(t)
                         -u^{\varepsilon_n}(t)\|_{H^1}
 +\|v^{\varepsilon_{n+1}}-v^{\varepsilon_n}\|_{X^1([0,k])}\\
 &\qquad\leq d_n(\varepsilon_n,\varepsilon_{n+1})
 \leq2^{-n}
 \end{split}
\]
for all sufficiently large \(n\) on this event. Therefore
\[
 \sum_{n\ge k}
 \left[
 \sup_{0\leq t\leq k}\|u^{\varepsilon_{n+1}}(t)
                         -u^{\varepsilon_n}(t)\|_{H^1}
 +\|v^{\varepsilon_{n+1}}-v^{\varepsilon_n}\|_{X^1([0,k])}
 \right]<\infty
\]
on that event.

The space of \(H^1\)-valued c\`adl\`ag functions on \([0,k]\) is
complete for the uniform distance: a uniform limit has right continuity
and left limits by completeness of \(H^1\).
The continuous remainder space \(X^1([0,k])\) is also complete.
We therefore obtain pathwise limits
\begin{align}
 u^{\varepsilon_n}&\longrightarrow u^{(k)}
     &&\text{uniformly in \(H^1\) on \([0,k]\)},\label{prob:asuniform}\\
 v^{\varepsilon_n}&\longrightarrow v^{(k)}
     &&\text{in \(X^1([0,k])\)}.\label{prob:ascritical}
\end{align}
On overlapping intervals the uniform limits agree, and the critical
limits agree as well because \(X^1\) embeds into \(C H^1\).
Thus they define processes \(u\) and \(v\) on the whole half-line.
The convolution convergence in \eqref{prob:subsequencechoice2} gives
\[
 u-v=\Psi
\]
simultaneously on all these intervals, proving \eqref{prob:limitclass}.
In particular, the limiting \(u\) has the required c\`adl\`ag paths and
the remainder \(v\) is continuous.

At every fixed time \(t\), \(u^{\varepsilon_n}(t)\) is
\(\mathcal F_t\)-measurable and converges almost surely in \(H^1\).
The limit is therefore adapted. Completeness of the filtration permits
the common exceptional null set to be removed consistently; on it one
may choose a fixed deterministic path with initial value \(u_0\).
An adapted right-continuous process with values in the separable space
\(H^1\) is progressively measurable.
Since \(v=u-\Psi\), the remainder is adapted; its continuity
also makes it progressively measurable.
For each fixed \(k\), Lemma~\ref{prob:measurable} places all
\(v^{\varepsilon_n}\) in the closed separable subspace \(\mathcal X_k\).
The critical limit is an \(\mathcal X_k\)-valued measurable random
variable. This also verifies measurability of the critical distances
to the limit.

It remains to identify the equation. For \(a,b\in H^1(\mathbb T^3)\),
the pointwise polynomial bound and Sobolev duality imply
\begin{equation}\label{prob:nonlinearHminus1}
 \begin{split}
 \|\mathcal N(a)-\mathcal N(b)\|_{H^{-1}}
 &\lesssim\|\mathcal N(a)-\mathcal N(b)\|_{L^{6/5}}\\
 &\lesssim\bigl(\|a\|_{L^6}^4+\|b\|_{L^6}^4\bigr)\|a-b\|_{L^6}\\
 &\lesssim\bigl(\|a\|_{H^1}^4+\|b\|_{H^1}^4\bigr)\|a-b\|_{H^1}.
 \end{split}
\end{equation}
On every path for which \eqref{prob:asuniform} holds, the sequence
\(u^{\varepsilon_n}\) is uniformly bounded in \(L^\infty([0,k];H^1)\).
Writing
\[
 B_k=\sup_n\sup_{0\leq t\leq k}
           \|u^{\varepsilon_n}(t)\|_{H^1}
          +\sup_{0\leq t\leq k}\|u(t)\|_{H^1}<\infty,
\]
the bound \eqref{prob:nonlinearHminus1} yields
\[
 \begin{split}
 &\int_0^k
   \|\mathcal N(u^{\varepsilon_n}(s))-\mathcal N(u(s))\|_{H^{-1}}
       \,\mathrm ds\\
 &\qquad\lesssim kB_k^4
       \sup_{0\leq s\leq k}
           \|u^{\varepsilon_n}(s)-u(s)\|_{H^1}
       \longrightarrow0.
 \end{split}
\]
Unitarity of \(S(t)\) on \(H^{-1}\) then gives
\[
 \begin{split}
 &\sup_{0\leq t\leq k}
 \left\|\int_0^tS(t-s)
   [\mathcal N(u^{\varepsilon_n}(s))-\mathcal N(u(s))]
              \,\mathrm ds\right\|_{H^{-1}}\\
 &\qquad\leq\int_0^k
   \|\mathcal N(u^{\varepsilon_n}(s))-\mathcal N(u(s))\|_{H^{-1}}
              \,\mathrm ds
 \longrightarrow0.
 \end{split}
\]
Passing to the limit in \eqref{prob:cutoffmild} yields
\[
 u(t)=S(t)u_0
       -i\int_0^tS(t-s)\mathcal N(u(s))\,\mathrm ds+\Psi(t)
\]
simultaneously for all \(t\in[0,k]\) as an identity in \(H^{-1}\).
Subtracting \(\Psi(t)\) gives
\[
 -i\int_0^tS(t-s)\mathcal N(u(s))\,\mathrm ds
       =v(t)-S(t)u_0.
\]
The right-hand side belongs to \(X^1([0,k])\) by
\eqref{prob:ascritical} and the free evolution estimate
\eqref{det:linear}. Thus the integral has the continuous \(H^1\)
representative required in Definition~\ref{def:solution}.

To prove convergence of the entire cutoff family, put
\[
 d_T(\varepsilon)=
 \sup_{0\leq t\leq T}\|u^\varepsilon(t)-u(t)\|_{H^1}
 +\|v^\varepsilon-v\|_{X^1([0,T])}.
\]
For fixed \(\varepsilon\),
\(d_T(\varepsilon,\varepsilon_n)\to d_T(\varepsilon)\) almost surely.
For \(\rho>0\), this convergence implies
\[
 \mathbf1_{\{d_T(\varepsilon)>\rho\}}
 \leq\liminf_{n\to\infty}
       \mathbf1_{\{d_T(\varepsilon,\varepsilon_n)>\rho/2\}}
 \quad\hbox{almost surely}.
\]
Fatou's lemma and \eqref{prob:cauchyquantitative} therefore give
\begin{equation}\label{prob:limitprobabilitybound}
 \begin{split}
 \mathbb P(d_T(\varepsilon)>\rho)
 &\leq\liminf_{n\to\infty}
       \mathbb P(d_T(\varepsilon,\varepsilon_n)>\rho/2)\\
 &\leq\frac{C_{T,u_0,G,\nu}}{R^2}
        +\frac{16TK_T(R)^2}{\rho^2}q_3(\varepsilon),
 \end{split}
\end{equation}
where \(\varepsilon_n<\varepsilon\) for all sufficiently large \(n\).
For every fixed \(R\), \(q_3(\varepsilon)\to0\) gives
\[
 \limsup_{\varepsilon\downarrow0}
          \mathbb P(d_T(\varepsilon)>\rho)
 \leq C_{T,u_0,G,\nu}R^{-2}.
\]
Letting \(R\to\infty\) proves \eqref{prob:allcutoffs}.

Finally, \(F\) is locally Lipschitz on \(H^1\). More precisely, its
polynomial expression and Sobolev embedding give
\[
 |F(a)-F(b)|
 \lesssim\bigl(1+\|a\|_{H^1}^5+\|b\|_{H^1}^5\bigr)\|a-b\|_{H^1}.
\]
Thus \eqref{prob:asuniform} implies uniform convergence of
\(F(u^{\varepsilon_n})\) to \(F(u)\) on \([0,T]\), pathwise.
Consequently, Fatou's lemma and Proposition~\ref{prob:energy}
give
\[
 \begin{split}
 \mathbb E\sup_{0\leq t\leq T}F(u(t))^2
 &\leq\liminf_{n\to\infty}
       \mathbb E\sup_{0\leq t\leq T}F(u^{\varepsilon_n}(t))^2\\
 &\leq C_TF(u_0)^2,
 \end{split}
\]
which is \eqref{prob:limitenergy}. By
\(\|w\|_{H^1}^4\leq4F(w)^2\),
\begin{equation}\label{prob:fourthH1moment}
 \mathbb E\sup_{0\leq t\leq T}\|u(t)\|_{H^1}^4
 \le4C_TF(u_0)^2.
\end{equation}
This completes the proof.
\end{proof}

\subsection{Pathwise uniqueness and continuous dependence}

The construction also identifies the jumps. The Hilbert-space Poisson
integral in the interaction representation satisfies
\(\Delta_tY=S(-t)G\,\Delta_tL\). Since \(S(t)\) is strongly continuous,
\(\Delta_t\Psi=G\,\Delta_tL\). The continuity of \(v\) therefore gives,
outside a single null set,
\begin{equation}\label{prob:limitingjumps}
 \Delta_tu=G\,\Delta_tL,\qquad t>0.
\end{equation}
\begin{proposition}\label{prob:uniqueness}
Two solutions in Definition~\ref{def:solution} with the same initial
datum and the same Poisson random measure on the prescribed stochastic
basis are indistinguishable.
\end{proposition}

\begin{proof}
Let \(u,\widetilde u\) be two such solutions and set
\(v=u-\Psi\), \(\widetilde v=\widetilde u-\Psi\).
Fix an integer \(T>0\), and take a common event of probability one
on which both remainders belong to \(X^1([0,T])\), both full solutions
are c\`adl\`ag in \(H^1\), and \(\Psi\) is c\`adl\`ag in \(H^3\).
Every c\`adl\`ag Banach-valued function on a compact interval is bounded.
Thus, for each path in this event, there are finite numbers \(R\geq1\)
and \(K\geq0\) with
\[
 \|v\|_{L^\infty([0,T];H^1)}
 +\|\widetilde v\|_{L^\infty([0,T];H^1)}\le R,
 \qquad
 \|\Psi\|_{L^\infty([0,T];H^3)}\le K.
\]
Theorem~\ref{det:stability} applied to this path gives
\[
 \begin{split}
 \|v-\widetilde v\|_{X^1([0,T])}
 &\leq C_T(R,K)
       \bigl(\|v(0)-\widetilde v(0)\|_{H^1}
                  +\|\Psi-\Psi\|_{Z([0,T])}\bigr)\\
 &=0.
 \end{split}
\]
The continuous embedding into \(C([0,T];H^1)\) implies equality at every
time, and hence \(u=\widetilde u\) on \([0,T]\).
Intersecting the probability-one events over integer \(T\) proves
indistinguishability on \([0,\infty)\).
This completes the proof.
\end{proof}

\begin{proposition}
\label{prob:continuity}
Let \(u_{0,n}\to u_0\) in \(H^1\), and construct \(u_n,u\) using the
same Poisson random measure and the coefficient \(G\).
Then, for every \(T<\infty\),
\begin{equation}\label{prob:initialcontinuity}
 \sup_{0\leq t\leq T}\|u_n(t)-u(t)\|_{H^1}
 +\|(u_n-\Psi)-(u-\Psi)\|_{X^1([0,T])}
 \longrightarrow0
 \quad\hbox{in probability}.
\end{equation}
\end{proposition}

\begin{proof}
The function \(F\) is continuous on \(H^1\), so
\(\sup_nF(u_{0,n})<\infty\).
The energy estimate \eqref{prob:limitenergy} is therefore uniform
in \(n\), with the same coefficient-dependent constant.
Let \(\mathcal B_R^n\) be the event where
\[
 \sup_{0\leq t\leq T}F(u_n(t))\le R,\qquad
 \sup_{0\leq t\leq T}F(u(t))\le R,\qquad
 \sup_{0\leq t\leq T}\|\Psi(t)\|_{H^3}\le R.
\]
The union bound and Markov's inequality yield
\[
 \begin{split}
 \mathbb P((\mathcal B_R^n)^c)
 &\leq R^{-2}
 \left(\mathbb E\sup_{0\leq t\leq T}F(u_n(t))^2
       +\mathbb E\sup_{0\leq t\leq T}F(u(t))^2\right)\\
 &\quad+R^{-2}
       \mathbb E\sup_{0\leq t\leq T}\|\Psi(t)\|_{H^3}^2\\
 &\leq R^{-2}\bigl(C_TF(u_{0,n})^2+C_TF(u_0)^2+4Tq_3\bigr).
 \end{split}
\]
Thus
\begin{equation}\label{prob:databadset}
 \sup_n\mathbb P((\mathcal B_R^n)^c)
 \leq \frac{C_{T,G,\nu,\{u_{0,n}\},u_0}}{R^2}.
\end{equation}
On \(\mathcal B_R^n\), both remainders have \(H^1\) bounds controlled
only by \(R\). Since the forcing is the same, deterministic stability
gives
\[
 \|(u_n-\Psi)-(u-\Psi)\|_{X^1([0,T])}
 \le C_T(R)\|u_{0,n}-u_0\|_{H^1}.
\]
Denote the left-hand side of \eqref{prob:initialcontinuity} by
\(D_{T,n}\). Since \(u_n-u=(u_n-\Psi)-(u-\Psi)\), the embedding
\eqref{det:linear} gives
\[
 D_{T,n}\leq2C_T(R)\|u_{0,n}-u_0\|_{H^1}
 \quad\hbox{on }\mathcal B_R^n.
\]
It follows that, for every \(\rho>0\),
\[
 \begin{split}
 \mathbb P(D_{T,n}>\rho)
 &\leq\mathbb P((\mathcal B_R^n)^c)
    +\mathbf1_{\{2C_T(R)\|u_{0,n}-u_0\|_{H^1}>\rho\}},\\
 \limsup_{n\to\infty}\mathbb P(D_{T,n}>\rho)
 &\leq C_{T,G,\nu,\{u_{0,n}\},u_0}R^{-2}.
 \end{split}
\]
Letting \(R\to\infty\) proves \eqref{prob:initialcontinuity}.

For measurability, apply Lemma~\ref{prob:measurable} separately
to \(u_0\) and each \(u_{0,n}\). The closed linear span of the
resulting countable collection of separable subspaces is separable.
All limiting remainders are measurable random elements of that
space, by their construction as almost-sure \(X^1\) limits.
Their \(X^1\) distances are therefore measurable, as are the
uniform distances between the c\`adl\`ag solutions.
This completes the proof.
\end{proof}

\begin{proof}[Proof of Theorem~\ref{thm:main}]
Proposition~\ref{prob:cutoff} constructs the finite-activity solutions.
Propositions~\ref{prob:cauchy} and \ref{prob:limit} construct their
global limit on the given stochastic basis and identify the mild
equation. The class assertion and cutoff convergence follow from
\eqref{prob:limitclass} and \eqref{prob:allcutoffs}.
Propositions~\ref{prob:energy} and \ref{prob:limit} give
\[
 \sup_{0<\varepsilon<1}\mathbb E\sup_{0\leq t\leq T}
               F(u^\varepsilon(t))^2
 +\mathbb E\sup_{0\leq t\leq T}F(u(t))^2
 \leq2C_TF(u_0)^2,
\]
which is \eqref{eq:main-moment}.
Proposition~\ref{prob:uniqueness} proves pathwise uniqueness in the
stated class, and Proposition~\ref{prob:continuity} proves continuous
dependence.
This completes the proof.
\end{proof}

\section{Further stability consequences}\label{sec:consequences}

The moment bound \eqref{eq:main-moment} upgrades the convergence
in \eqref{eq:main-cutoff} and \eqref{eq:main-data} to convergence
in mean in the uniform energy path norm.

\subsection{Convergence in mean in the energy path norm}

\begin{corollary}\label{cor:Lq}
For every $T<\infty$ and $0<q<4$,
\begin{equation}\label{eq:Lq-cutoff}
 \lim_{\varepsilon\downarrow0}
 \mathbb E\left[\sup_{0\leq t\leq T}
        \|u^\varepsilon(t)-u(t)\|_{H^1}^{q}\right]=0.
\end{equation}
If $u_{0,n}\to u_0$ in $H^1$ and the solutions use the same
noise, then
\begin{equation}\label{eq:Lq-data}
 \lim_{n\to\infty}
 \mathbb E\left[\sup_{0\leq t\leq T}
        \|u_n(t)-u(t)\|_{H^1}^{q}\right]=0.
\end{equation}
\end{corollary}

\begin{proof}
Set $e_\varepsilon=\sup_{0\leq t\leq T}
\|u^\varepsilon(t)-u(t)\|_{H^1}$. From
$\|w\|_{H^1}^2\le2F(w)$ and
Theorem~\ref{thm:main}(ii),
\begin{equation}\label{eq:fourth-bound}
 \sup_{0<\varepsilon<1}\mathbb E e_\varepsilon^4
 \lesssim\sup_{0<\varepsilon<1}
       \mathbb E\sup_{0\leq t\leq T}
            \bigl(F(u^\varepsilon(t))^2+F(u(t))^2\bigr)
 <\infty.
\end{equation}
Write $B_T=\sup_{0<\varepsilon<1}\mathbb E e_\varepsilon^4<\infty$.
For $\delta>0$, split the expectation according to whether
$e_\varepsilon\leq\delta$. H\"older's inequality with exponents
$4/q$ and $4/(4-q)$ gives
\begin{align*}
 \mathbb E e_\varepsilon^q
 &\leq \delta^q+
  \mathbb E\bigl[e_\varepsilon^q
                    \mathbf1_{\{e_\varepsilon>\delta\}}\bigr]\\
 &\leq \delta^q+B_T^{q/4}
              \mathbb P(e_\varepsilon>\delta)^{1-q/4}.
\end{align*}
By \eqref{eq:main-cutoff}, the probability tends to zero for every
fixed $\delta>0$. Therefore
$\limsup_{\varepsilon\downarrow0}\mathbb E e_\varepsilon^q
\leq\delta^q$, and letting $\delta\downarrow0$ proves
\eqref{eq:Lq-cutoff}.
For the initial-data assertion, the sequence $u_{0,n}$ is
bounded in $H^1$, and the embedding $H^1\hookrightarrow L^6$
implies $\sup_nF(u_{0,n})<\infty$. The same fourth-moment
bound gives
\[
 \sup_n\mathbb E\left[
   \sup_{0\leq t\leq T}\|u_n(t)-u(t)\|_{H^1}^4\right]<\infty.
\]
Applying the preceding H\"older estimate with
$e_n=\sup_{0\leq t\leq T}\|u_n(t)-u(t)\|_{H^1}$ and using
\eqref{eq:main-data} proves \eqref{eq:Lq-data}.
This completes the proof.
\end{proof}

\subsection{A finite-time vanishing-noise limit}

For $0\le\lambda\le1$, let $u^\lambda$ solve the equation
with coefficient $\lambda G$ and initial datum $u_0$ on the
same stochastic basis. Its stochastic convolution is
$\lambda\Psi$. Let $w$ be the global deterministic solution
with datum $u_0$, so that $u^0=w$.

\begin{corollary}\label{cor:vanishing}
For each $T<\infty$,
\begin{equation}\label{eq:vanishing-prob}
 \sup_{0\leq t\leq T}\|u^\lambda(t)-w(t)\|_{H^1}
    +\|u^\lambda-\lambda\Psi-w\|_{X^1([0,T])}
           \longrightarrow0
 \quad\hbox{in probability as }\lambda\downarrow0.
\end{equation}
For each $0<q<4$, one also has
\begin{equation}\label{eq:vanishing-Lq}
 \lim_{\lambda\downarrow0}
 \mathbb E\left[\sup_{0\leq t\leq T}
              \|u^\lambda(t)-w(t)\|_{H^1}^q\right]=0.
\end{equation}
\end{corollary}

\begin{proof}
For $0\leq\lambda\leq1$, set
\[
 h_\lambda(w,z)=F(w+\lambda Gz)-F(w),\qquad
 R_\lambda(w,z)=h_\lambda(w,z)-\lambda DF(w)[Gz].
\]
The Taylor formula used in Lemma~\ref{prob:taylor} gives
\[
 R_\lambda(w,z)
 =\lambda^2\int_0^1(1-s)
       D^2F(w+s\lambda Gz)[Gz,Gz]\,ds.
\]
Since $0\leq\lambda\leq1$ and $|z|\leq1$, its estimates yield
\[
 0\leq R_\lambda(w,z)\lesssim_G\lambda^2|z|^2F(w),\qquad
 |h_\lambda(w,z)|\lesssim_G\lambda|z|F(w).
\]
After integration in $z$, the drift is bounded by
$C_G\lambda^2m_2F(w)$ and the bracket density by
$C_G\lambda^2m_2F(w)^2$. The proof of
Proposition~\ref{prob:energy} thus applies with constants
independent of $\lambda$, giving
\begin{equation}\label{eq:lambda-moment}
 \sup_{0\le\lambda\le1}\mathbb E\sup_{0\leq t\leq T}
                              F(u^\lambda(t))^2
                  \le C_TF(u_0)^2.
\end{equation}
Moreover,
\begin{equation}\label{eq:lambda-psi}
 \mathbb E\sup_{0\leq t\leq T}\|\lambda\Psi(t)\|_{H^3}^2
                               \le4T\lambda^2q_3.
\end{equation}

Fix $R$ large enough that $F(w(t))=F(u_0)\le R$ and set
\[
 A_R^\lambda=
 \left\{\sup_{0\leq t\leq T}F(u^\lambda(t))\le R,
        \ \sup_{0\leq t\leq T}\|\lambda\Psi(t)\|_{H^3}\le R\right\}.
\]
Markov's inequality and
\eqref{eq:lambda-moment}--\eqref{eq:lambda-psi} give
\[
 \mathbb P((A_R^\lambda)^c)
 \leq\frac{C_TF(u_0)^2+4T\lambda^2q_3}{R^2}.
\]
Hence there is
$C_{\mathrm{err}}=C_{\mathrm{err}}(T,u_0,G,m_2)<\infty$ such that
$\mathbb P((A_R^\lambda)^c)\le C_{\mathrm{err}}/R^2$, uniformly in
$0\le\lambda\le1$ and $R\geq F(u_0)$. On $A_R^\lambda$, the remainders
$v^\lambda=u^\lambda-\lambda\Psi$ and $w$ have bounded
$L^\infty H^1$ norms, and their external paths are bounded
in $L^\infty H^3$. The deterministic stability estimate
of Section~\ref{sec:deterministic} gives
\[
 \|v^\lambda-w\|_{X^1([0,T])}
       \le C_T(R)\lambda
                      \sup_{0\leq t\leq T}\|\Psi(t)\|_{H^3}.
\]
The embedding into $C_tH^1_x$ and the triangle inequality
give the same bound, with an enlarged constant, for the sum
of distances in \eqref{eq:vanishing-prob}. Therefore, for
$a>0$,
\begin{equation}\label{eq:lambda-prob-bound}
 \mathbb P\left(
    \sup_{0\leq t\leq T}\|u^\lambda(t)-w(t)\|_{H^1}
       +\|v^\lambda-w\|_{X^1([0,T])}>a\right)
  \le\frac{C_{\mathrm{err}}}{R^2}
       +\frac{4TC_T(R)^2q_3}{a^2}\lambda^2.
\end{equation}
For fixed $R$, taking the upper limit as $\lambda\downarrow0$
in \eqref{eq:lambda-prob-bound} leaves at most
$C_{\mathrm{err}}/R^2$. Letting $R\to\infty$ proves
\eqref{eq:vanishing-prob}. Put
$e_\lambda=\sup_{0\leq t\leq T}\|u^\lambda(t)-w(t)\|_{H^1}$.
By \eqref{eq:lambda-moment}, mass and energy conservation for
$w$, and $\|a\|_{H^1}^2\leq2F(a)$,
\[
 B_T:=\sup_{0\leq\lambda\leq1}\mathbb E e_\lambda^4<\infty.
\]
For $0<q<4$ and $\delta>0$, H\"older's inequality now gives
\[
 \mathbb E e_\lambda^q
 \leq\delta^q+B_T^{q/4}
                 \mathbb P(e_\lambda>\delta)^{1-q/4}.
\]
Letting $\lambda\downarrow0$ and then $\delta\downarrow0$
proves \eqref{eq:vanishing-Lq}. This completes the proof.
\end{proof}

Corollary~\ref{dyn:nonuniform-zero-noise} gives the corresponding
long-time separation from the deterministic flow.

\section{Mass escape and absence of free scattering}
\label{sec:dynamics}

The mass balance yields an expected exit-time bound and
almost-sure escape from bounded mass sets. These estimates also
give lower bounds for the expected energy and obstructions to
free scattering.

\subsection{Pathwise mass balance and escape from bounded mass sets}

For $w\in L^2$ and $z\in Z_0$, set
\[
 b_0=\|G\|_{\mathcal L(\mathbb R^m,L^2(\mathbb T^3))},
 \qquad
 h_M(w,z)=2\operatorname{Re}\langle w,Gz\rangle_{L^2}
                          +\|Gz\|_{L^2}^2.
\]
Thus $\|Gz\|_2\le b_0$ for $z\in Z_0$. As before,
$q_0=\int_{Z_0}\|Gz\|_2^2\,\nu(dz)$.

\begin{proposition}
\label{dyn:massbalance}\label{prop:mean-mass}
The solution of Theorem~\ref{thm:main} satisfies, almost surely
simultaneously for all $t\ge0$,
\begin{equation}\label{eq:pathwise-mass}
 \|u(t)\|_2^2
  =\|u_0\|_2^2+q_0t+\mathcal M_t,
 \qquad
 \mathcal M_t=
 \int_{(0,t]}\int_{Z_0}
       h_M(u(s-),z)\,\widetilde N(ds,dz).
\end{equation}
The process $\mathcal M$ is a square-integrable martingale on
each finite time interval. In particular,
\begin{equation}\label{eq:mean-mass}
 \mathbb E\|u(t)\|_2^2=\|u_0\|_2^2+q_0t,
 \qquad t\ge0.
\end{equation}
\end{proposition}

\begin{proof}
For $w,g\in L^2$, expansion of the square gives
\[
 \|w+g\|_2^2-\|w\|_2^2
  =2\operatorname{Re}\langle w,g\rangle_{L^2}+\|g\|_2^2.
\]
Mass conservation between retained jumps therefore yields
\[
 \|u^\varepsilon(t)\|_2^2-\|u_0\|_2^2
  =\int_{(0,t]}\int_{Z_\varepsilon}
       h_M(u^\varepsilon(s-),z)\,N(ds,dz).
\]
The symmetry identity \eqref{prob:zerodrift} gives
\begin{align*}
 \int_{Z_\varepsilon}h_M(w,z)\,\nu(dz)
  &=2\operatorname{Re}\left\langle w,
       \int_{Z_\varepsilon}Gz\,\nu(dz)\right\rangle_{L^2}
       +\int_{Z_\varepsilon}\|Gz\|_2^2\,\nu(dz)\\
  &=\int_{Z_\varepsilon}\|Gz\|_2^2\,\nu(dz).
\end{align*}
Subtracting the compensator gives
\begin{equation}\label{eq:mass-finite}
 \begin{split}
 \|u^\varepsilon(t)\|_2^2
  ={}&\|u_0\|_2^2
       +t\int_{Z_\varepsilon}\|Gz\|_2^2\,\nu(dz)\\
     &+\int_{(0,t]}\int_{Z_\varepsilon}
          h_M(u^\varepsilon(s-),z)\,\widetilde N(ds,dz).
 \end{split}
\end{equation}
For any $w\in L^2$, Cauchy--Schwarz and $(a+b)^2\leq2a^2+2b^2$
give
\[
 |h_M(w,z)|^2
 \leq8\|w\|_2^2\|Gz\|_2^2+2\|Gz\|_2^4
 \leq(8\|w\|_2^2+2b_0^2)\|Gz\|_2^2.
\]
Integration in $z$ yields
\begin{equation}\label{eq:mass-bracket-bound}
 \int_{Z_0}|h_M(w,z)|^2\,\nu(dz)
   \le (8\|w\|_2^2+2b_0^2)q_0.
\end{equation}
The right-hand side has an integrable time integral when
$w=u(s-)$, by Theorem~\ref{thm:main}(ii). The martingale
$\mathcal M$ in \eqref{eq:pathwise-mass} is therefore well defined
and square integrable on $[0,T]$ for each $T<\infty$.

To pass to the limit in the martingale term, write
$\mathcal M^\varepsilon$ for the last term of
\eqref{eq:mass-finite}. On $Z_\varepsilon$ the difference
of its integrand and that of $\mathcal M$ equals
$2\operatorname{Re}\langle u^\varepsilon(s-)-u(s-),Gz\rangle$.
On $Z_0\setminus Z_\varepsilon$ the omitted integrand is
$h_M(u(s-),z)$. Since these two mark sets are disjoint, the
Poisson isometry and Doob's inequality give
\begin{align*}
 \mathbb E\sup_{0\leq t\leq T}
       |\mathcal M_t^\varepsilon-\mathcal M_t|^2
 &\leq4\mathbb E\int_0^T\int_{Z_\varepsilon}
   \bigl|2\operatorname{Re}\langle
       u^\varepsilon(s-)-u(s-),Gz\rangle_{L^2}\bigr|^2
            \,\nu(dz)\,ds\\
 &\quad+4\mathbb E\int_0^T\int_{Z_0\setminus Z_\varepsilon}
        |h_M(u(s-),z)|^2\,\nu(dz)\,ds.
\end{align*}
Using the preceding pointwise bounds, we obtain
\begin{align}
 \mathbb E\sup_{0\leq t\leq T}|\mathcal M_t^\varepsilon-\mathcal M_t|^2
 &\le 16Tq_0\,
       \mathbb E\sup_{0\leq s\leq T}\|u^\varepsilon(s)-u(s)\|_2^2
       \notag\\
 &\quad+4T\left(8\mathbb E\sup_{0\leq s\leq T}\|u(s)\|_2^2
                          +2b_0^2\right)q_3(\varepsilon)
 \longrightarrow0.
 \label{eq:mass-martingale-limit}
\end{align}
Here $\|Gz\|_2\le\|Gz\|_{H^3}$ and
Corollary~\ref{cor:Lq} with $q=2$ were used. Moreover,
\begin{align*}
 \mathbb E\sup_{0\leq t\leq T}
     \bigl|\|u^\varepsilon(t)\|_2^2-\|u(t)\|_2^2\bigr|
 &\leq
 \left(\mathbb E\sup_{0\leq t\leq T}
                 \|u^\varepsilon(t)-u(t)\|_2^2\right)^{1/2}\\
 &\quad\times
 \left(\mathbb E\sup_{0\leq t\leq T}
          (\|u^\varepsilon(t)\|_2+\|u(t)\|_2)^2\right)^{1/2}
 \longrightarrow0.
\end{align*}
The first factor tends to zero by Corollary~\ref{cor:Lq}, and the
second is uniformly bounded by \eqref{eq:main-moment}.
The drift in \eqref{eq:mass-finite} converges uniformly on $[0,T]$
to $q_0t$. Thus \eqref{eq:pathwise-mass} follows as an identity of
c\`adl\`ag processes on $[0,T]$. Taking a common event of
probability one for integer $T$ proves the simultaneous assertion
on $[0,\infty)$. Since $\mathbb E\mathcal M_t=0$, taking
expectations gives \eqref{eq:mean-mass}. This completes the proof.
\end{proof}

\begin{proposition}\label{dyn:massescape}
Assume that the effective noise is nonzero, in the sense that
$q_0>0$. For $R>\|u_0\|_2$, define
\[
 \tau_R=\inf\{t\ge0:\|u(t)\|_2>R\},
 \qquad \inf\varnothing=\infty.
\]
Then
\begin{equation}\label{eq:mass-exit-bound}
 \mathbb E\tau_R
  \le\frac{(R+b_0)^2-\|u_0\|_2^2}{q_0}<\infty.
\end{equation}
Consequently,
\begin{equation}\label{eq:mass-limsup}
 \mathbb P\left(
      \limsup_{t\to\infty}\|u(t)\|_2=\infty\right)=1.
\end{equation}
\end{proposition}

\begin{proof}
For every $t>0$, right continuity gives
\[
 \{\tau_R<t\}
  =\bigcup_{r\in\mathbb Q\cap[0,t)}\{\|u(r)\|_2>R\}
  \in\mathcal F_t.
\]
The right continuity of the filtration then implies
$\{\tau_R\leq t\}\in\mathcal F_t$, so $\tau_R$ is a stopping
time. Since $\|u_0\|_2<R$,
right continuity at zero gives $\tau_R>0$ on every path under
consideration. On $\{\tau_R<\infty\}$, the definition of $\tau_R$
gives times $s_n\in[\tau_R,\tau_R+1/n)$ with $\|u(s_n)\|_2>R$.
Right continuity then yields $\|u(\tau_R)\|_2\geq R$, while the
left limit satisfies $\|u(\tau_R-)\|_2\leq R$.
To bound the jump at $\tau_R$, consider the nested sets
$Z_{1/k}$, $k\geq2$. For every integer $j\geq1$,
$N$ has finitely many atoms on $(0,j]\times Z_{1/k}$, since
$j\nu(Z_{1/k})\leq jk^2m_2$. Conditional on their number, the
time coordinates have a joint density and are therefore distinct
almost surely. Intersecting these events over $j$ and $k$ gives
one event on which no two atoms of $N$ share a time coordinate:
any two marks belong to some common $Z_{1/k}$ and any two finite
times belong to some common $(0,j]$. On its intersection with
the event in \eqref{prob:limitingjumps}, every nonzero jump of
$u$ equals $Gz$ for a single mark $z\in Z_0$. Consequently
$\|\Delta_tu\|_2\leq b_0$ for every $t>0$, and
\[
 R\leq\|u(\tau_R)\|_2
 \leq\|u(\tau_R-)\|_2+\|\Delta_{\tau_R}u\|_2
 \leq R+b_0
 \quad\hbox{on }\{\tau_R<\infty\}.
\]
For $0\leq t<\tau_R$, the norm is at most $R$. Thus, for every $t\geq0$,
\begin{equation}\label{eq:mass-overshoot}
 \|u(t\wedge\tau_R)\|_2\le R+b_0
 \quad\hbox{almost surely}.
\end{equation}
The stopped integral $\mathcal M_{t\wedge\tau_R}$ is a true
square-integrable martingale: its predictable integrand is
$\mathbf1_{\{0<s\leq\tau_R\}}h_M(u(s-),z)$, and
\eqref{eq:mass-bracket-bound} gives
\[
 \mathbb E\int_0^t\mathbf1_{\{s\leq\tau_R\}}
       \int_{Z_0}|h_M(u(s-),z)|^2\,\nu(dz)\,ds
 \leq t(8R^2+2b_0^2)q_0<\infty.
\]
Stopping
\eqref{eq:pathwise-mass}, taking expectations, and using
\eqref{eq:mass-overshoot} give
\[
 \|u_0\|_2^2+q_0\mathbb E(t\wedge\tau_R)
   =\mathbb E\|u(t\wedge\tau_R)\|_2^2
   \le (R+b_0)^2.
\]
Monotone convergence as $t\to\infty$ proves
\eqref{eq:mass-exit-bound}. In particular, $\tau_R<\infty$
almost surely.

Choose integers $R_n\uparrow\infty$ with $R_1>\|u_0\|_2$.
On a common probability-one event, each $\tau_{R_n}$ is finite
and $\|u(\tau_{R_n})\|_2\geq R_n$.
Boundedness of c\`adl\`ag paths on compact time intervals implies
$\tau_{R_n}\to\infty$: otherwise a bounded time interval would
contain values of $\|u\|_2$ exceeding every $R_n$.
The values at these times prove \eqref{eq:mass-limsup}.
This completes the proof.
\end{proof}

\begin{corollary}
\label{dyn:energy-growth}
For every $t\ge0$,
\begin{equation}\label{eq:energy-growth-lower}
 \mathbb E E(u(t))\ge\frac{(M(u_0)+q_0t)^3}{6},
 \qquad
 \mathbb E F(u(t))^2\ge\frac{(M(u_0)+q_0t)^6}{36}.
\end{equation}
If $q_0>0$, then also
$\limsup_{t\to\infty}E(u(t))=\infty$ almost surely.
\end{corollary}

\begin{proof}
Since Haar measure has total mass one, H\"older's inequality gives
\[
 M(w)=\int_{\mathbb T^3}|w|^2\,dx
 \leq\left(\int_{\mathbb T^3}|w|^6\,dx\right)^{1/3}
       \left(\int_{\mathbb T^3}1\,dx\right)^{2/3}
 =\|w\|_6^2.
\]
Consequently,
\[
 E(w)\geq\frac16\|w\|_6^6\geq\frac16M(w)^3,
 \qquad
 F(w)^2\geq E(w)^2\geq\frac1{36}M(w)^6.
\]
The moment bound \eqref{eq:main-moment} implies integrability of
$F(u(t))^2$, and hence of $M(u(t))^6$. Jensen's inequality and
\eqref{eq:mean-mass} therefore yield
\begin{align*}
 \mathbb E E(u(t))
 &\geq\frac16\mathbb E M(u(t))^3
  \geq\frac16\bigl(\mathbb E M(u(t))\bigr)^3
  =\frac16(M(u_0)+q_0t)^3,\\
 \mathbb E F(u(t))^2
 &\geq\frac1{36}\mathbb E M(u(t))^6
  \geq\frac1{36}\bigl(\mathbb E M(u(t))\bigr)^6
  =\frac1{36}(M(u_0)+q_0t)^6.
\end{align*}
If $q_0>0$, the pointwise inequality
$E(u(t))\geq M(u(t))^3/6$ and \eqref{eq:mass-limsup} imply the
asserted almost-sure energy divergence along a sequence of times.
This completes the proof.
\end{proof}

\subsection{Scattering obstructions on the torus}
\label{subsec:scattering-obstructions}

For $s\in\{0,1\}$, we say that $u$ \emph{scatters freely in $H^s$}
if there exists $u_+\in H^s$ such that
\begin{equation}\label{dyn:free-scattering}
 \lim_{t\to\infty}\|u(t)-S(t)u_+\|_{H^s}=0.
\end{equation}
The case $s=0$ is free scattering in $L^2$.
The absence of nontrivial asymptotically linear solutions for the
periodic cubic equation was discussed by Colliander, Keel,
Staffilani, Takaoka and Tao~\cite[Appendix]{CKSTT2010}.
The following proposition treats the deterministic quintic equation.

\begin{proposition}
\label{dyn:det-nonscattering}
Let $w\in C([0,\infty);H^1)$ be a global solution of
\[
 \partial_t w=i(\Delta w-|w|^4w)
\]
in the deterministic energy solution class. If there exists
$w_+\in H^1$ for which
$\|w(t)-S(t)w_+\|_{H^1}\to0$ as $t\to\infty$, then
$w(0)=0$ and $w\equiv0$.
\end{proposition}

\begin{proof}
Write $\mu=M(w(0))=\|w(t)\|_2^2$. Conservation of mass and energy
gives $\sup_{t\ge0}\|w(t)\|_{H^1}<\infty$. Suppose, for a
contradiction, that $\mu>0$, and put $f(t)=S(t)w_+$.
We use the complex $L^2$ pairing, linear in its first entry, and set
\[
 A(t)=\int_{\mathbb T^3}S(-t)w(t,x)\,\overline{w_+(x)}\,dx.
\]
This scalar function is bounded by $\mu^{1/2}\|w_+\|_2$.
The mild equation in the interaction representation reads
\[
 S(-t)w(t)=w(0)-i\int_0^t S(-s)\mathcal N(w(s))\,ds
 \quad\hbox{in }H^{-1}.
\]
Since $\mathcal N:H^1\to H^{-1}$ is continuous and
$S(t)$ is strongly continuous on $H^{-1}$, the right-hand side
is continuously differentiable as an $H^{-1}$-valued map.
Pairing with the fixed element $w_+\in H^1$ is therefore legitimate,
and unitarity, first for smooth functions and then by duality, yields
\[
 A'(t)=-i\int_{\mathbb T^3}|w(t)|^4w(t)\,
                         \overline{f(t)}\,dx.
\]
All terms in this identity are integrable because
$w(t),f(t)\in L^6$. Consequently,
\begin{align*}
 \operatorname{Im}A'(t)
 &=-\|w(t)\|_6^6
   +\operatorname{Re}\int_{\mathbb T^3}|w(t)|^4w(t)
                         \overline{w(t)-f(t)}\,dx\\
 &\le-\|w(t)\|_6^6
      +\|w(t)\|_6^5\|w(t)-f(t)\|_6.
\end{align*}
Sobolev embedding and the uniform $H^1$ bound give
\[
 \|w(t)\|_6^5\|w(t)-f(t)\|_6
 \lesssim \left(\sup_{s\geq0}\|w(s)\|_{H^1}\right)^5
                         \|w(t)-f(t)\|_{H^1}
 \longrightarrow0.
\]
Since Haar measure is normalized to have total
mass one, H\"older's inequality and mass conservation give
$\|w(t)\|_6^6\ge\|w(t)\|_2^6=\mu^3$.
Choose $T_0$ so that the error term is at most $\mu^3/2$ for
every $t\geq T_0$. Then, for $t\geq T_0$,
\[
 \operatorname{Im}A(t)
 =\operatorname{Im}A(T_0)
       +\int_{T_0}^t\operatorname{Im}A'(s)\,ds
 \leq\operatorname{Im}A(T_0)-\frac{\mu^3}{2}(t-T_0).
\]
The right-hand side tends to $-\infty$, contradicting
$|A(t)|\leq\mu^{1/2}\|w_+\|_2$. Hence $\mu=0$ and mass
conservation gives $w\equiv0$. This completes the proof.
\end{proof}

\begin{corollary}
\label{dyn:stoch-nonscattering}
Assume $q_0>0$. On one event of probability one, there is no
$u_+\in L^2(\mathbb T^3)$, including a sample-dependent choice,
such that
\[
 \|u(t)-S(t)u_+\|_2\longrightarrow0\qquad(t\to\infty).
\]
In fact, $u$ cannot converge in $L^2$ to any trajectory that is
bounded in $L^2$ on a tail interval. This excludes asymptotic
convergence to any deterministic unforced energy solution, also
when its initial state is allowed to depend on the sample.
\end{corollary}

\begin{proof}
Use the single full-probability event in
Proposition~\ref{dyn:massescape} on which
$\limsup_{t\to\infty}\|u(t)\|_2=\infty$.
Fix a sample in this event. If $a$ is a trajectory with
$B:=\sup_{t\geq T_0}\|a(t)\|_2<\infty$ and
$\|u(t)-a(t)\|_2\to0$, choose $T_1\geq T_0$ so that the
distance is at most one for every $t\geq T_1$. Then
\[
 \sup_{t\geq T_1}\|u(t)\|_2
 \leq\sup_{t\geq T_1}\|u(t)-a(t)\|_2
       +\sup_{t\geq T_1}\|a(t)\|_2
 \leq1+B,
\]
contradicting the defining property of the event. A free orbit
$a(t)=S(t)u_+$ has norm $\|u_+\|_2$, and a deterministic
unforced energy solution has constant mass. These arguments
hold for every target trajectory on the chosen event, including
targets chosen after the sample is fixed. This completes the proof.
\end{proof}

\begin{corollary}
\label{dyn:jump-obstruction}
Assume $q_0>0$. On one event of probability one, no continuous
path $a:[0,\infty)\to L^2$ satisfies
\[
 \|u(t)-a(t)\|_2\longrightarrow0\qquad(t\to\infty).
\]
The path $a$ may depend on the sample and may be unbounded.
\end{corollary}

\begin{proof}
Since $q_0>0$, the set $\{z\in Z_0:Gz\ne0\}$ has positive
$\nu$-measure. It is the union of the sets
$\{z\in Z_0:\|Gz\|_2>1/n\}$, $n\geq1$. Hence there is $r>0$
such that
\[
 A_r=\{z\in Z_0:\|Gz\|_2>r\},\qquad
 0<\nu(A_r)\leq r^{-2}q_0<\infty.
\]
The events $\{N((j,j+1]\times A_r)\geq1\}$, $j\geq0$,
are independent and each has probability $1-e^{-\nu(A_r)}>0$.
By the second Borel--Cantelli lemma, infinitely many of them
occur almost surely. The corresponding arrival times can
therefore be enumerated as $t_n\to\infty$. Intersect this event
with the one constructed in the proof of
Proposition~\ref{dyn:massescape}, on which distinct Poisson atoms
have distinct time coordinates and \eqref{prob:limitingjumps}
holds at every time. At $t_n$ there is exactly one mark $z_n$,
and $z_n\in A_r$. Hence
\[
 \|u(t_n)-u(t_n-)\|_2>r\qquad(n\geq1).
\]
If a continuous path $a$ satisfied the asserted convergence,
then for every $\delta>0$ there would be $T$ such that
$\|u(t)-a(t)\|_2<\delta$ for all $t\geq T$.
For $t_n>T$, choose $s_k\uparrow t_n$ with $s_k\geq T$.
The left limit of $u$ and the continuity of $a$ give
\[
 \|u(t_n-)-a(t_n)\|_2
 =\lim_{k\to\infty}\|u(s_k)-a(s_k)\|_2\leq\delta.
\]
The triangle inequality then yields
\[
 r<\|u(t_n)-u(t_n-)\|_2
 \leq\|u(t_n)-a(t_n)\|_2
       +\|a(t_n)-u(t_n-)\|_2
 \leq2\delta,
\]
which is impossible for $\delta<r/2$. Since the event was chosen
before $a$, the conclusion holds simultaneously for every
continuous path $a$. This completes the proof.
\end{proof}

\begin{corollary}
\label{dyn:nonuniform-zero-noise}
Assume $q_0>0$, and let $u^\lambda$ and $w$ be as in
Corollary~\ref{cor:vanishing}. For each fixed $0<\lambda\leq1$,
\[
 \limsup_{t\to\infty}\|u^\lambda(t)-w(t)\|_2=\infty
 \quad\hbox{almost surely}.
\]
\end{corollary}

\begin{proof}
For each fixed $0<\lambda\leq1$, the coefficient $\lambda G$
satisfies Assumption~\ref{ass:noise} and has positive mass
injection rate $\lambda^2q_0$. Proposition~\ref{dyn:massescape}
therefore gives $\limsup_{t\to\infty}\|u^\lambda(t)\|_2=\infty$
almost surely. Mass conservation for $w$ and the reverse
triangle inequality give, for every $t\geq0$,
\[
 \|u^\lambda(t)-w(t)\|_2
 \geq\|u^\lambda(t)\|_2-\|w(t)\|_2
 =\|u^\lambda(t)\|_2-\|u_0\|_2.
\]
Taking the upper limit as $t\to\infty$ proves the assertion.
This completes the proof.
\end{proof}

\section*{Data availability}
No datasets were generated or analysed during this study.

\section*{Declaration of AI assistance}
During the preparation of this manuscript, the authors used ChatGPT
as an auxiliary tool for language editing, organization of arguments, proof
auditing, and manuscript preparation. All mathematical statements, proofs, references,
and conclusions appearing in the manuscript were independently verified by
the authors, who take full responsibility for the content.



\begin{thebibliography}{99}

\bibitem{Applebaum2009}D. Applebaum,
\emph{L\'evy Processes and Stochastic Calculus}, second edition,
Cambridge Studies in Advanced Mathematics, vol.~116,
Cambridge University Press, Cambridge, 2009.

\bibitem{BRZ2014}V. Barbu, M. R\"ockner and D. Zhang,
Stochastic nonlinear Schr\"odinger equations with linear multiplicative noise: rescaling approach,
\emph{J. Nonlinear Sci.} \textbf{24} (2014), no.~3, 383--409.

\bibitem{BRZ2016}V. Barbu, M. R\"ockner and D. Zhang,
Stochastic nonlinear Schr\"odinger equations,
\emph{Nonlinear Anal.} \textbf{136} (2016), 168--194.

\bibitem{Bourgain1993}J. Bourgain,
Fourier transform restriction phenomena for certain lattice subsets and applications to nonlinear evolution equations. Part I: Schr\"odinger equations,
\emph{Geom. Funct. Anal.} \textbf{3} (1993), no.~2, 107--156.

\bibitem{BD2015}J. Bourgain and C. Demeter,
The proof of the $\ell^2$ decoupling conjecture,
\emph{Ann. of Math. (2)} \textbf{182} (2015), no.~1, 351--389.

\bibitem{BHM2020}Z. Brze\'{z}niak, F. Hornung and U. Manna,
Weak martingale solutions for the stochastic nonlinear Schr\"odinger equation driven by pure jump noise,
\emph{Stoch. Partial Differ. Equ. Anal. Comput.} \textbf{8} (2020), no.~1, 1--53.

\bibitem{BHW2019}Z. Brze\'{z}niak, F. Hornung and L. Weis,
Martingale solutions for the stochastic nonlinear Schr\"odinger equation in the energy space,
\emph{Probab. Theory Related Fields} \textbf{174} (2019), nos.~3--4, 1273--1338.

\bibitem{BHW2022}Z. Brze\'{z}niak, F. Hornung and L. Weis,
Uniqueness of martingale solutions for the stochastic nonlinear Schr\"odinger equation on 3d compact manifolds,
\emph{Stoch. Partial Differ. Equ. Anal. Comput.} \textbf{10} (2022), no.~3, 828--857.

\bibitem{BLZ2021}Z. Brze\'{z}niak, W. Liu and J. Zhu,
The stochastic Strichartz estimates and stochastic nonlinear Schr\"odinger equations driven by L\'evy noise,
\emph{J. Funct. Anal.} \textbf{281} (2021), no.~4, Paper No.~109021.

\bibitem{BM2014}Z. Brze\'{z}niak and A. Millet,
On the stochastic Strichartz estimates and the stochastic nonlinear Schr\"odinger equation on a compact Riemannian manifold,
\emph{Potential Anal.} \textbf{41} (2014), no.~2, 269--315.

\bibitem{BGT2004}N. Burq, P. G\'erard and N. Tzvetkov,
Strichartz inequalities and the nonlinear Schr\"odinger equation on compact manifolds,
\emph{Amer. J. Math.} \textbf{126} (2004), no.~3, 569--605.

\bibitem{Cazenave2003}T. Cazenave,
\emph{Semilinear Schr\"odinger Equations},
Courant Lecture Notes in Mathematics, vol.~10,
Courant Institute of Mathematical Sciences, New York, and American Mathematical Society, Providence, RI, 2003.

\bibitem{CLO2021}K. Cheung, G. Li and T. Oh,
Almost conservation laws for stochastic nonlinear Schr\"odinger equations,
\emph{J. Evol. Equ.} \textbf{21} (2021), 1865--1894.

\bibitem{CM2019}K. Cheung and R. Mosincat,
Stochastic nonlinear Schr\"odinger equations on tori,
\emph{Stoch. Partial Differ. Equ. Anal. Comput.} \textbf{7} (2019), no.~2, 169--208.

\bibitem{CKSTT2008}J. Colliander, M. Keel, G. Staffilani, H. Takaoka and T. Tao,
Global well-posedness and scattering for the energy-critical nonlinear Schr\"odinger equation in $\mathbb{R}^3$,
\emph{Ann. of Math. (2)} \textbf{167} (2008), no.~3, 767--865.

\bibitem{CKSTT2010}J. Colliander, M. Keel, G. Staffilani,
H. Takaoka and T. Tao,
Transfer of energy to high frequencies in the cubic defocusing
nonlinear Schr\"odinger equation,
\emph{Invent. Math.} \textbf{181} (2010), 39--113.

\bibitem{dBD1999}A. de Bouard and A. Debussche,
A stochastic nonlinear Schr\"odinger equation with multiplicative noise,
\emph{Comm. Math. Phys.} \textbf{205} (1999), no.~1, 161--181.

\bibitem{dBD2003}A. de Bouard and A. Debussche,
The stochastic nonlinear Schr\"odinger equation in $H^1$,
\emph{Stoch. Anal. Appl.} \textbf{21} (2003), no.~1, 97--126.

\bibitem{dBH2019}A. de Bouard and E. Hausenblas,
The nonlinear Schr\"odinger equation driven by jump processes,
\emph{J. Math. Anal. Appl.} \textbf{475} (2019), no.~1, 215--252.

\bibitem{dBHO2019}A. de Bouard, E. Hausenblas and M. Ondrej\'at,
Uniqueness of the nonlinear Schr\"odinger equation driven by jump processes,
\emph{NoDEA Nonlinear Differential Equations Appl.} \textbf{26} (2019), no.~3, Paper No.~22.

\bibitem{FX2021}C. Fan and W. Xu,
Global well-posedness for the defocussing mass-critical stochastic nonlinear Schr\"odinger equation on $\mathbb{R}$ at $L^2$ regularity,
\emph{Anal. PDE} \textbf{14} (2021), no.~8, 2561--2594.

\bibitem{FXWZ2021}C. Fan and W. Xu,
A Wong--Zakai theorem for the stochastic mass-critical nonlinear Schr\"odinger equation,
\emph{SIAM J. Math. Anal.} \textbf{53} (2021), no.~3, 3681--3715.

\bibitem{HHK2009}M. Hadac, S. Herr and H. Koch,
Well-posedness and scattering for the KP-II equation in a critical space,
\emph{Ann. Inst. H. Poincar\'e Anal. Non Lin\'eaire} \textbf{26} (2009), no.~3, 917--941.

\bibitem{HHK2010}M. Hadac, S. Herr and H. Koch,
Erratum to ``Well-posedness and scattering for the KP-II equation in a critical space'',
\emph{Ann. Inst. H. Poincar\'e Anal. Non Lin\'eaire} \textbf{27} (2010), no.~3, 971--972.

\bibitem{HRZ2019}S. Herr, M. R\"ockner and D. Zhang,
Scattering for stochastic nonlinear Schr\"odinger equations,
\emph{Comm. Math. Phys.} \textbf{368} (2019), no.~2, 843--884.

\bibitem{HTT2011}S. Herr, D. Tataru and N. Tzvetkov,
Global well-posedness of the energy-critical nonlinear Schr\"odinger equation with small initial data in $H^1(\mathbb{T}^3)$,
\emph{Duke Math. J.} \textbf{159} (2011), no.~2, 329--349.

\bibitem{Hornung2018}F. Hornung,
The nonlinear stochastic Schr\"odinger equation via stochastic Strichartz estimates,
\emph{J. Evol. Equ.} \textbf{18} (2018), no.~3, 1085--1114.

\bibitem{IP2012}A. D. Ionescu and B. Pausader,
The energy-critical defocusing NLS on $\mathbb{T}^3$,
\emph{Duke Math. J.} \textbf{161} (2012), no.~8, 1581--1612.

\bibitem{KeelTao1998}M. Keel and T. Tao,
Endpoint Strichartz estimates,
\emph{Amer. J. Math.} \textbf{120} (1998), no.~5, 955--980.

\bibitem{KV2016}R. Killip and M. Vi\c{s}an,
Scale invariant Strichartz estimates on tori and applications,
\emph{Math. Res. Lett.} \textbf{23} (2016), no.~2, 445--472.

\bibitem{Kwak2026}B. Kwak,
Global well-posedness of the energy-critical nonlinear Schr\"odinger equations on $\mathbb{T}^d$,
\emph{Ann. PDE} \textbf{12} (2026), Paper No.~19.

\bibitem{LOT2025}G. Li, M. Okamoto and L. Tao,
Global well-posedness of the energy-critical stochastic nonlinear Schr\"odinger equation on the three-dimensional torus,
preprint, arXiv:2407.15108v2 (2025; first posted 2024).

\bibitem{LXYZ2026}B. Liu, F. Xu, B. Yang and L. Zhang,
Global well-posedness and scattering for energy-critical Schr\"odinger equations with conservative Marcus noise,
preprint, arXiv:2609.37534 (2026).

\bibitem{MR2026}A. Millet and S. Roudenko,
Well-posedness of the focusing stochastic nonlinear Schr\"odinger equation: $L^2$-critical and supercritical cases,
\emph{Stoch. Partial Differ. Equ. Anal. Comput.} (2026), published online 23 September 2026.

\bibitem{OO2020}T. Oh and M. Okamoto,
On the stochastic nonlinear Schr\"odinger equations at critical regularities,
\emph{Stoch. Partial Differ. Equ. Anal. Comput.} \textbf{8} (2020), 869--894.

\bibitem{PeszatZabczyk2007}S. Peszat and J. Zabczyk,
\emph{Stochastic Partial Differential Equations with L\'evy Noise: An Evolution Equation Approach},
Encyclopedia of Mathematics and its Applications, vol.~113,
Cambridge University Press, Cambridge, 2007.

\bibitem{Strunk2015}N. Strunk,
Global well-posedness of the energy-critical defocusing NLS on rectangular tori in three dimensions,
\emph{Differential Integral Equations} \textbf{28} (2015), nos.~11--12, 1069--1084.

\bibitem{Tao2006}T. Tao,
\emph{Nonlinear Dispersive Equations: Local and Global Analysis},
CBMS Regional Conference Series in Mathematics, vol.~106,
American Mathematical Society, Providence, RI, 2006.

\bibitem{WZZ2023}J. Wang, J. Zhai and J. Zhu,
The stochastic nonlinear Schr\"odinger equations driven by pure jump noise,
\emph{Statist. Probab. Lett.} \textbf{197} (2023), Paper No.~109810.

\bibitem{Zhang2022}D. Zhang,
Strichartz and local smoothing estimates for stochastic dispersive equations with linear multiplicative noise,
\emph{SIAM J. Math. Anal.} \textbf{54} (2022), no.~6, 5981--6017.

\bibitem{Zhang2023}D. Zhang,
Stochastic nonlinear Schr\"odinger equations in the defocusing mass and energy critical cases,
\emph{Ann. Appl. Probab.} \textbf{33} (2023), no.~5, 3652--3705.

\bibitem{ZLZ2026}J. Zhu, W. Liu and J. Zhai,
Large deviation principles for stochastic nonlinear Schr\"odinger equations driven by L\'evy noise,
\emph{J. Funct. Anal.} \textbf{290} (2026), no.~9, Paper No.~111377.

\end{thebibliography}
\end{document}